\documentclass[reqno,12pt]{amsart}

\usepackage{graphics,amssymb,amsthm,amsmath}
\usepackage[english]{babel}
\usepackage{fullpage}
\usepackage{enumerate}
\usepackage{hyperref}
\usepackage{soul}

\usepackage{multirow}
\numberwithin{equation}{section}

\newcommand{\I}{\mathbb{I}}
\newcommand{\IZ}{\ensuremath{\mathbb{Z}}}

\newcommand{\nn}{\nonumber}

\newcommand{\mathsym}[1]{{}}

\newcommand{\half}{\frac{1}{2}}
\newcommand{\be}{\begin{equation}}
\newcommand{\ee}{\end{equation}}

\newcommand{\balpha}{{\boldsymbol\alpha}}

\newcommand{\co}{\mbox{\footnotesize\textrm{c}}_I}
\newcommand{\cd}{\mbox{\footnotesize\textrm{c}}_{II}}

\newtheorem{teo}{Theorem}[section]
\newtheorem{prop}[teo]{Proposition}
\newtheorem{cor}[teo]{Corollary}
\newtheorem{rem}[teo]{Remark}
\newtheorem{defi}[teo]{Definition}
\newtheorem{lemma}[teo]{Lemma}

\usepackage{color}

\title{Quantized Matrix Algebras and Quantum seeds}
\date{\today}
\author{Hans Plesner Jakobsen, Chiara Pagani}
\address[]{\textit{Hans P. Jakobsen} \newline \indent 
University of Copenhagen,
 Department of Mathematical Sciences
\newline \indent Universitetsparken
5, DK-2100, Copenhagen, Denmark
}
\email{jakobsen@math.ku.dk}

\address[]{\textit{Chiara Pagani} \newline \indent   University of Luxembourg, 
Mathematics Research Unit,\newline \indent
6, rue Richard Coudenhove-Kalergi,
L-1359 Luxembourg,
Grand-Duchy of Luxembourg
}
\email{chiara.pagani@uni.lu}

\begin{document}

\begin{abstract}
We determine explicit quantum seeds for classes of quantized matrix algebras.
Furthermore, we obtain results on centers and block diagonal forms {of these
algebras.} In the case where $q$ is {an arbitrary}  root of unity, this
further determines the degrees.
\end{abstract}
\maketitle



\section{Introduction}

The class of  $q$-deformations of the coordinate algebras  of simple matrix Lie
groups derived from {so-called} FRT bialgebras 
 is of great importance in quantum group theory.  These bialgebras arise as
suitable extensions and quotients of the coordinate algebra 
 ${\mathcal O}_q(M(n))$ of the quantum matrix space  and have been a topic of
interest ever since they were introduced by 
 L. D. Faddeev,  N. Yu.  Reshetikhin,  and L. A. Takhtadzhyan in 1990
\cite{frt}.  We will refer 
to  ${\mathcal O}_q(M(n))$ as the  FRT algebra.
Other quantizations of the coordinate algebras of matrix groups have been 
proposed and studied, not the least the quantum deformation $A_q(n)(d^{-1})$ of
general
 linear groups introduced by R. Dipper and S. Donkin in \cite{dd} as the
localization of a quantum version $A_q(n)$ of the coordinate algebra of $n
\times n$ matrices at a non-central element $d$, the $q$-determinant. The FRT
and Dipper-Donkin quantum algebras share some common properties, for instance
the same classical limit as the parameter of deformation $q$ goes to $1$.  They
are also related in another way, as we will explain in Section \ref{se:setup}. 
In other aspects, however, e.g. P.I. degrees, they are different (see 
\cite{jz1}).

In 2001, S. Fomin and A. Zelevinsky introduced  a new class of commutative rings under the name of cluster algebras \cite{fz1}.  A cluster algebra is generated by a set of generators called the cluster variables; clusters are not given at first but are obtained from an initial one via a process of mutations.  One of their
main motivations for introducing this new class of algebras was to
provide an algebraic framework for studying total positivity in semisimple
groups  and canonical bases for quantum groups.  
  Later, in \cite{bz}, A. Berenstein and A. Zelevinsky 
introduced quantum deformations of cluster algebras pertaining to a notion of canonical basis in cluster algebras.

The theory of cluster algebras has developed  vastly in recent years. In particular with the advent of quantum cluster algebras it became a challenge to relate the above mentioned coordinate algebras of matrix groups to that theory, notably to  write down so-called quantum seeds for these, thus describing their possible cluster algebra structure. 
Right from the beginning some general and strong results were obtained \cite{bz}. Later, a method has been developed by H. P. Jakobsen  and H. Zhang \cite{jz2} targeting directly the FRT algebras. The approach in \cite{bz} has recently been extended in \cite{gls}.

In this paper we construct quantum seeds associated to quantum matrix algebras $\mathcal{M}_q$ belonging to a certain family, which includes in particular the FRT algebra ${\mathcal O}_q(M(n,r))$ and the Dipper-Donkin algebra $A_q(n,r)$.  (These algebras were originally defined as coordinate algebras of $n \times n$ matrices, but we will here extend the defining relations to arbitrary $n \times r$ matrices while maintaining the notation.) For other deformations of the matrix algebras, see \cite{jz3}, and references cited therein.
 For the purposes of the present article, one can view these algebras as being defined over ${\mathbb C}$, with $q$ being a so-called `dummy' parameter, leaving open even the possibility of $q$ being a root of unity, though the latter case will not be pursued much  here.  Each quantum matrix algebra $\mathcal{M}_q$ has an associated quasi-polynomial algebra $\overline{{\mathcal M}}_q$ whose generators have commutation relations given in terms of an integer skew-symmetric matrix $H_\mathcal{M}$ (see Definition \ref{def:H})  which will play a central role in the determination of the initial clusters.  

\noindent 
The complicated nature of the classification already on the level of cluster algebras (\cite{fz1, fz2})  has been intimidating  for explicit constructions of large scale examples of, say, initial seeds (cf. below). Indeed, only very few low dimensional examples were computed (\cite{gl}, \cite{gls}).  In this sense it is a surprise that  elementary operations can be used to obtain far reaching and very explicit conclusions. In particular, to construct an initial seed for ${\mathcal O}_q(M(n,r))$ both in the guise of  the FRT as for  the Dipper-Donkin algebra. 
Indeed, the method of Gauss Elimination on the level of the matrix blocks of $H_\mathcal{M}$ leads to both results pertaining to the case of $q$ a root of unity, viz. the degree of algebras,  and to explicit compatible pairs  $(\Lambda_{\mathcal M}, B_{\mathcal M})$ for a family of quantized matrix algebras. Some detailed information about the $q$ commutation relations between quantized minors in the quantized matrix algebra and how they can be computed by considering the $q$ commutation relations between the diagonals of these minors, but now computed in the associated quasi-polynomial algebras, are needed too. To wit, this connection allows us to perform a simple change-of-basis operation on $H_{\mathcal M}$ whereby $\Lambda_{\mathcal M}$ is obtained. Later, some elementary algebra is needed too.
In the end, the results themselves are not only very explicit but also very striking. The structural `defining matrices'  $\Lambda_{\mathcal M}$  are integer $nr \times nr$ matrices and as their sizes increase, there are no bounds on the integers that may occur.  Nevertheless, in case they are invertible (at least in the cases considered), their inverses have entries from $\{-1,0,1\}$ or $\{-2,0,2\}$.
\\
As a consequence of our approach it turns out that we can, with a little extra
effort, determine the centers {of the underlying quantized matrix algebras}.
Further results relate to block diagonal forms. This leads further to a complete
determination of the degrees of specialized versions of the algebras when $q$ is
an arbitrary root of unity.
\\

More precisely, the structure of the paper  is the following.  In Section \ref{se:setup} we present the algebras $\mathcal{M}_q$ as certain subalgebras of the algebra ${\mathcal P}_q={\mathcal P}_q(n,r)$. This latter was originally introduced and studied in \cite{jz1}: it is a semidirect product
\begin{equation*}{\mathcal P}_q(n,r)={\mathcal O}_q(M(n,r)\times_s{\mathcal
L}[R_1,\dots,R_n,C_1,\dots,C_r] \subset{\mathcal U}_q({\mathfrak
g}_{n+r}),
\end{equation*} 
where ${\mathcal L}[R_1,\dots,R_n,C_1,\dots,C_r]$ denotes the algebra of Laurent
polynomials in $n$ `row operators' $R_\alpha$ and $r$ `column operators' $C_j$
 that come directly  from a  quantized Cartan subalgebra in ${\mathcal U}_q({\mathfrak
g}_{n+r})$. One fundamental
assumption  is that 
\begin{equation*} 
 {\mathcal O}_q(M(n,r))\times_s{\mathcal L}[R_1,\dots,R_n,C_1,\dots,C_r]={\mathcal M}_q\times_s{\mathcal L}[R_1,\dots,R_n,C_1,\dots,C_r],
\end{equation*}
but other more technical assumptions are also imposed to ensure that
quantized minors can be defined in ${\mathcal M}_q$. Indeed, we introduce a
family ${\mathcal V}^+_{\mathcal M}$ of $nr$  minors  $\chi_{\alpha j} \in
{\mathcal M}_q$ which are quasi-commuting, i.e. $\chi_{\alpha j}\chi_{\beta
t}=q^{\Lambda_{\alpha j,\beta
t}}\chi_{\beta t}\chi_{\alpha j}.$ The \textit{quasi-commutation matrix} $\Lambda_{\mathcal{M}}$ which encodes their commutation relations is related to the matrix $H_{\mathcal M}$ via 
\begin{equation*}
 \Lambda_{\mathcal M}={\mathbb T}^tH_{\mathcal M}{\mathbb T},
\end{equation*}
 where  ${\mathbb T}$ is an explicitly given matrix. (See Proposition~\ref{211}.)
\\
We finish Section \ref{se:setup} by introducing our object of interest: quantum seeds and compatible pairs $(\Lambda_{\mathcal M}, B_{\mathcal M})$   associated to a family of minors ${\mathcal V}^+_{\mathcal M}$. 

In Section \ref{se:Hmatrix} the Gauss Elimination is carried out in block form for some matrices $H_{\mathcal M}$,  including the ones associated to the FRT and the  Dipper-Donkin  algebras, as well as for   ${\mathcal P}_q(n,r)$. The resulting upper triangular block matrix is sufficiently well determined that one can read off the determinant, the rank, and in case of invertibility, the blocks of the inverse. Many results of \cite{jj} and  \cite{jz1} are extended, while others are given elementary proofs. This Section also carries the first version of the very striking form of the entries of the inverse matrix $H_\mathcal{M}^{-1}$.

The results of Section \ref{se:Hmatrix} are put to further use in Section \ref{se:degrees} where block diagonal forms as well as  degrees of the various algebras are determined.

Section \ref{se:Linv} deals with the quasi-commutation matrices $\Lambda_{\mathcal
M}$ for some relevant  quantum matrix algebras of our family. Again we can very
explicitly give the inverse matrices $\Lambda_{\mathcal M}^{-1}$ (when they
exist) and observe the striking forms they take. In case of positive co-rank,
the kernels can also be determined very explicitly. This
furthermore carries immediate results for the centers, both generically, as well as  in the case where $q$ is a
primitive root of unity.

Finally, in Section \ref{se:cpairs}, it is explained how the results of Section \ref{se:Linv} carry all relevant information for the determination of compatible pairs.

\bigskip

\textit{Notation.} 
All  algebras we will consider are over the field ${\mathbb C}$. We will use the
notation $\I$ for the identity matrix of any order. {When needed, we will indicate with a subscript its rank.   If not specified
differently, indices  in lower case Greek letters $\alpha,\beta, \dots$ run from
$1,\dots n$, while lower-case Latin letters $i,j,\dots$ indicate indices running
from $1$ to $r$. Here $r,~n$ are two fixed positive integers.


\section{The general set-up} \label{se:setup}

Let $q$ be a fixed non-zero complex
number.
The coordinate algebra  ${\mathcal O}_q(M(n,r))$ of the quantum
$n\times r$ matrix space is the associative algebra generated by elements
$Z_{\alpha j},\alpha=1,2,\cdots, n$, $j=1,2,\cdots,r$, subject to the following  
defining relations:
\begin{eqnarray}\label{relations1}
Z_{\alpha j}Z_{\alpha k}&=&qZ_{\alpha k}Z_{
\alpha j} \; ,\quad
j<k, \nn \\
 Z_{\alpha j}Z_{\beta j}&=&qZ_{\beta j}Z_{\alpha j} \; ,\quad \alpha <\beta
,\nn \\
Z_{\alpha j}Z_{\beta l}&=&Z_{\beta l}Z_{\alpha j}\; ,\quad \alpha >\beta ,
j<l,\nn \\
Z_{\alpha j}Z_{\beta l}&=&Z_{\beta l}Z_{\alpha j}+(q-q^{-1})Z_{\alpha l}Z_{\beta
j} \; ,\quad 
\alpha <\beta , j<l.\end{eqnarray} 
The quantized matrix algebras  ${\mathcal O}_q(M(n,r))$  were introduced (in the case $n=r$) by Faddeev,
Reshetikhin and  Takhtadzhyan in  \cite{frt}. We will  refer to the
general algebra  ${\mathcal O}_q(M(n,r))$ as the FRT algebra.
\\

Following \cite[p. 58]{pdc}   we  consider the associated
quasi-polynomial algebra $\overline{\mathcal O_q(M(n,r))}$ of ${\mathcal O}_q(M(n,r))$. 
This is the  associative algebra generated by elements
$z_{\alpha j},\alpha =1,2,\cdots, n$, $j=1,2,\cdots,r$, subject to the following
defining relations:
\begin{eqnarray}\label{relations2}
z_{\alpha j}z_{\alpha k}&=&qz_{\alpha
k}z_{\alpha j} \; ,\quad
j<k, \nn \\
 z_{\alpha j}z_{\beta j}&=&qz_{\beta j}z_{\alpha j} \; ,\quad \alpha <\beta
,\nn \\
z_{\alpha j}z_{\beta l}&=&z_{\beta l}z_{\alpha j}\: , \quad \text{ otherwise}.
\end{eqnarray}

As in \cite{jz1} we introduce a
new quantum algebra  ${\mathcal P}_q=\mathcal{P}_q(n,r)$
containing  ${\mathcal O}_q(M(n,r))$ while having additional, mutually
commuting, generators $R_\alpha,~R_\alpha ^{-1}$, $\alpha =1,\dots, n$, and 
$C_j,~C_j^{-1}$, $j=1,\dots, r$
with the following additional relations
\begin{eqnarray}\label{relZRC}
R_\alpha ^{\pm1}Z_{\beta i}=q^{\pm\delta_{\alpha
,\beta }}Z_{\beta i}R_\alpha ^{\pm1} \; ,  &&  
C_j^{\pm1}Z_{\alpha i}=q^{\pm\delta_{j,i}}Z_{\alpha i}C_j^{\pm1},\nn \\
R_\alpha  R_\alpha ^{-1}=1 \; , && 
C_jC_j^{-1}=1
\end{eqnarray}
for all $\alpha, \beta=1,\dots ,n$ and $i,j=1, \dots, r$.
(As shown in \cite{jz1}, this is related to a quantization of a parabolic  subalgebra of $su(n,r)$.) We let
${\mathcal P}_q^+$ denote the subalgebra of ${\mathcal P}_q$ generated by the generators $Z_{\alpha i}$ of
${\mathcal O}_q(M(n,r))$ together with the elements $R_\alpha , \alpha =1,\dots,
n$ and
$C_j, j=1,\dots,r$.

The algebra ${\mathcal P}_q$ contains a central element ${\mathcal
Z}=\prod_\alpha 
R_\alpha  \prod_jC_j^{-1}$. This redundancy is preserved for convenience.

Recall  that in ${\mathcal O}_q(M(n,r))$ one can define  a `bar' involution $p\mapsto
\overline{p}$ as the unique  ${\mathbb C}$-linear algebra anti-automorphism such that:
\begin{equation*}\bar{q}=q^{-1}\textrm{ and } \overline{Z_{\alpha j}}=Z_{\alpha
j} \textrm{
for all } \alpha ,j.
\end{equation*}
The bar anti-automorphism is used in the definition of the (dual) canonical
basis. See e.g. \cite{jz3}.

\begin{defi}We extend the bar operation  to a linear anti-automorphism of
${\mathcal P}_q$ by the extra stipulation:
\begin{equation*}\forall \alpha:\overline{R_\alpha }=R_\alpha \textrm{ and }\forall
j:\overline{C_j}=C_j.
\end{equation*}
\end{defi}

\medskip 
\subsection{The general family}\label{gen-fam}
We wish to introduce and study a class of quantized matrix algebras from the
above. \\For each $\beta =1,\dots,n$ and each $i=1,\cdots,r$, fix $M_{\beta i}$
to be a
monomial in the generators $\{ R_\alpha ^{\pm1}, ~C_j^{\pm1} \}_{\alpha =1 \dots
n,j=1,\dots  r}$ (we occasionally
suppress the range of the indices when it is clear).  Once  such monomials 
$M_{\beta i}$ have  been fixed, there exist integers 
$\Phi_{\alpha j}^{\beta i}$ uniquely determined by \eqref{relZRC} as
\begin{equation}
M_{\alpha j}Z_{\beta i}=q^{\Phi_{\alpha j}^{\beta i}}Z_{\beta i}M_{\alpha j}.
\end{equation}
\noindent
We assume that $\forall \alpha =1,\cdots,n,\forall
j=1,\cdots,r:$
\begin{equation}\label{first} M_{\alpha j}M_{\alpha
+\gamma,j+k}=M_{\alpha,j+k}M_{\alpha+\gamma,j}\textrm{ for all
admissible  }  \gamma,k\in{\mathbb N}.\end{equation}
This implies that $\forall \alpha =1,\cdots,n,\forall j=1,\cdots,r:$
\begin{equation}\label{bottom-rel}
\Phi_{\alpha
j}^{\beta i}+\Phi_{\alpha +\gamma,j+k}^{\beta i}=\Phi_{\alpha
+\gamma,j}^{\beta i}+\Phi_{\alpha ,j+k}^{\beta i}\textrm{
for all admissible }\gamma,k\in \mathbb N.
\end{equation}

We will also need some relations between the symbols $ \Phi_{\alpha j}^{\beta i}$
that allow 
us to do specific computations and at the same time implies the following
identity
\begin{equation}\label{top-rel}
\Phi_{\alpha
j}^{\beta i}+\Phi_{\alpha j}^{\beta +\gamma, i+k}=\Phi_{\alpha j}^{\beta
+\gamma , i}+\Phi_{\alpha j}^{\beta , i+k}\textrm{ for all admissible }\gamma,k\in{\mathbb N},
\end{equation}
valid $\forall \alpha =1,\cdots,n,\forall
j=1,\cdots,r$.
Actually, the equations (\ref{bottom-rel}) and (\ref{top-rel}) together are
equivalent to 
\begin{equation*}\forall \alpha,\beta,i,j:
 \Phi_{\alpha j}^{\beta i}=a^\beta_\alpha+b^\beta_j+c^i_\alpha+d^i_j,
\end{equation*}
for
some appropriate integer-valued functions $a,b,c$, and $d$. This, on the
other hand, is clearly equivalent to the assumption which we now impose:
\begin{equation}
{\bf Assumption:}\qquad \forall \alpha,j: ~~M_{\alpha j}={\mathcal R}_\alpha
{\mathcal R}_j {\mathcal C}_\alpha {\mathcal C}_j,\qquad\qquad
\end{equation}
where ${\mathcal R}_x$ and ${\mathcal C}_y$ denote momomials in the row,
respectively column, operators.

All together, these conditions easily  imply

\begin{lemma} Let $\alpha _1,\dots , \alpha _{s} \in \{1, \dots n\}$ be pairwise different and let 
$j_1,\dots , j_{s} \in \{1, \dots r\}$ be pairwise different. The integer $\Psi$ in\label{psi} 
$$q^{\Psi}Z_{\alpha _1,j_{\sigma(1)}}Z_{\alpha _2,j_{\sigma(2)}}\cdots
Z_{\alpha
_s,j_{\sigma(s)}}=M_{\beta t}Z_{\alpha _1,j_{\sigma(1)}}Z_{\alpha
_2,j_{\sigma(2)}}\cdots
Z_{\alpha _s,j_{\sigma(s)}}M_{\beta t}^{-1},$$
while in general depending on $\beta ,t,\alpha _1,j_1,\dots,\alpha _s,j_s$, does
not depend on 
$\sigma\in S_s$.
Likewise, the integer $\Phi$ in 
$$q^{\Phi}Z_{\beta t}=\left(M_{\alpha _1,j_{\sigma(1)}}M_{\alpha
_2,j_{\sigma(2)}}\cdots
M_{\alpha
_s,j_{\sigma(s)}}\right)Z_{\beta t}\left(M_{\alpha
_1,j_{\sigma(1)}}M_{\alpha _2,j_{\sigma(2)
} } \cdots M_{\alpha _s,j_{\sigma(s)}}\right)^{-1},$$
while in general depending on $\beta ,t,\alpha _1,j_1,\dots,\alpha_s,j_s$, does
not depend on 
$\sigma\in S_s$.
\end{lemma}

\medskip

We finally add the following condition to our list of assumptions:
\begin{equation}\forall \alpha ,
j:\Phi^{\alpha j}_{\alpha j}=0.\label{third}
\end{equation}

(This could also be deduced from the stronger assumption that for no $x$
does ${\mathcal R}_x$ contain $R_x$, with a similar assumption for
the monomials ${\mathcal C}_y$. However, we shall not pursue these matter
further here.)

\bigskip

\begin{defi}The quantum algebra ${\mathcal M}_q $  is  the subalgebra of ${\mathcal P}_q$ generated by
the elements $W_{\alpha j}:=Z_{\alpha j}M_{\alpha j}$,  $\alpha =1,\dots,n$,
$j=1,\dots,r$.\end{defi}

The first condition (\ref{first}) guarantees that the relations of ${\mathcal
M}_q$ are similar to those of the FRT quantized matrix algebra. The use
of the second condition (\ref{top-rel}) is through Lemma~\ref{psi}, while we observe that the third condition
(\ref{third}) implies
\begin{equation*}
\forall (\alpha,j)=(1,1),\cdots,(n,r):Z_{\alpha j}M_{\alpha
j}=M_{\alpha j}Z_{\alpha j}.
\end{equation*}

\medskip
\noindent
The following then is clear:

\begin{lemma}
For all $\alpha ,j: \overline{W_{\alpha j}}=W_{\alpha j}$.
\end{lemma}

\medskip

\subsubsection{The Dipper-Donkin quantized matrix algebra}

Let $\theta:{\mathbb Z}\mapsto\{0,1\}$ be the discrete Heaviside function defined as usual by
$\theta(z)=1\Leftrightarrow z>0$. Let us set {$ M_{\alpha j}:= R_{\alpha+1} \dots R_{n}C_{j+1}^{-1} \dots
C_r^{-1}$, so that}
$\Phi_{\alpha j}^{\beta i}=\theta(\beta -\alpha)-\theta(i-j)$. The resulting
quantized matrix algebra has relations
\begin{eqnarray}\label{relDD}
W_{\alpha j}W_{\beta k}&=&q^2W_{\beta k}W_{\alpha j}\qquad \beta >\alpha,\;
k\leq j
\nn \\
W_{\alpha j}W_{\alpha k}&=&W_{\alpha k}W_{\beta j }\qquad \forall \alpha, \forall j,k \nn
\\
W_{\alpha j}W_{\beta k}&=&W_{\beta k}W_{\alpha j}+(q^2-1)W_{\beta
 j}W_{\alpha k}\qquad  \beta >\alpha ,\; k>j.
\end{eqnarray}

\begin{defi}\label{dddef}The quadratic algebra generated by elements 
$W_{\alpha j}$ with  relations \eqref{relDD} is called the Dipper-Donkin
quantized matrix
algebra and will be denoted by $\mathcal{D}_{q^2}(M(n,r))$.
\end{defi}
\noindent
{This algebra was introduced  and studied in \cite{dd}, in case  $n=r$. It is traditionally defined using $q$  instead of $q^2$ in \eqref{relDD} (It was originally denoted by  $A_q(n)$).

\medskip

\subsection{Quantum minors} Let $ m\leq min\{n,r\}$.
Given two sets   $\balpha=\{\alpha_1,\alpha_2, \dots ,\alpha_m\}\subseteq\{1,\dots,n\}$ such that $\alpha_1 < \alpha_2<\dots <\alpha_m  $
and $\bold{j}=\{j_1,j_2,\dots,j_m\}\subseteq\{1,\dots,r\}$ such that $j_1<j_2<\dots <j_m$, one can define  
\begin{eqnarray}\label{minors}
 \xi^{\balpha}_\bold{j}(Z):=\xi^\balpha_\bold{j}&:=&\Sigma_{\sigma\in
S_m}(-q)^{\ell(\sigma)}Z_{\alpha_1,j_{\sigma(1)}}Z_{\alpha_2,j_{\sigma(2)}}
\cdots Z_{\alpha_m,j_{\sigma(m)}} \nn
\\
&=&\Sigma_{\tau\in
S_m}(-q)^{\ell(\tau)} Z_{\alpha_{\tau(1)},j_1}Z_{\alpha_{\tau(2)},j_2} \cdots
Z_{\alpha_{\tau(r)},j_m}.\end{eqnarray} The elements  $\xi^\balpha_\bold{j} \in {\mathcal O}_q(M(n,r))$ are called the quantum
$m$-minors.

In the above expression - now seen in $\mathcal{P}_q$ -  we can
replace each $Z_{\alpha j}$  by $W_{\alpha j}M_{\alpha
j}^{-1}$   and then collect
all the $M_{\alpha j}$ factors to the (say) right. By means of the first
assumption,
we get an expression 
\begin{equation}
\xi^\balpha_\bold{j}=\tilde\chi^\balpha_\bold{j}\left(M_{\alpha_1j_1}
\cdot\cdots\cdot M_{\alpha_m j_m}\right)^{-1}.
\end{equation}

Since $\overline{\xi^\balpha_\bold{j}}=\xi^\balpha_\bold{j}$ it follows that
$\overline{\tilde\chi^\balpha_\bold{j}}=\left(M_{\alpha_1j_1}\cdot\cdots\cdot
M_{\alpha_m
j_m}\right){\tilde\chi^\balpha_\bold{j}}\left(M_{\alpha_1j_1}\cdot\cdots\cdot
M_{\alpha_m j_m}\right)^{-1}$. This implies that we may write
$\tilde\chi^\balpha_\bold{j}=q^a\chi^\balpha_\bold{j}$ for some integer  $a$ in such a way that
$\chi^\balpha_\bold{j}$ is invariant under the bar operator.

\begin{defi}
We call the element $\chi^\balpha_\bold{j} \in {\mathcal M}_q$ the quantum minor  (for the configuration
 $\balpha=\{\alpha_1,\alpha_2, \dots ,\alpha_m\}\subseteq\{1,\dots,n\}$,  with $\alpha_1<\alpha_2< \dots < \alpha_m $, and
$\bold{j}=\{j_1,j_2, \dots,
j_m\}\subseteq\{1,\dots,r\}$ with $j_1<j_2< \dots <
j_m$). We may clearly write,  for some functions 
$\ell^{(1)}_{\mathcal
M}, \ell^{(2)}_{\mathcal
M}:S_m\mapsto{\mathbb Z}$:
\begin{eqnarray}\label{qminor}
 \chi^\balpha_\bold{j}(W)=\chi^\balpha_\bold{j}&=&\Sigma_{\sigma\in
S_m}(-q)^{\ell^{(1)}_{\mathcal
M}(\sigma)}W_{\alpha_1,j_{\sigma(1)}}W_{\alpha_2,j_{\sigma(2)}}
\cdots W_{\alpha_m,j_{\sigma(m)}} \nn \\ 
&=&\Sigma_{\tau\in
S_m}(-q)^{\ell^{(2)}_{\mathcal M}(\tau)} W_{\alpha_{\tau(1)},j_1}W_{\alpha_{\tau(2)},j_2}
\cdots 
W_{\alpha_{\tau(m)},j_m}, \nn\\
\overline{\chi^\balpha_\bold{j}}&=&\chi^\balpha_\bold{j}, \nn \\
\xi^\balpha_\bold{j}&=&q^{\alpha^\balpha_\bold{j}}\chi^\balpha_\bold{j}
\left(M_{\alpha_1j_1}\cdot\cdots\cdot
M_{\alpha_m j_m}\right)^{-1}, \nn \\
q^{2\alpha^\balpha_\bold{j}}\chi^\balpha_\bold{j} &=&
\left(M_{\alpha_1 j_1}\cdot\cdots\cdot
M_{\alpha_m j_m}\right)^{-1}\chi^\balpha_\bold{j}\left(M_{\alpha_1
j_1}\cdot\cdots\cdot
M_{\alpha_m j_m}\right)
\end{eqnarray}
The last equations, which follow easily
from the above, are inserted for the sake of \S2.5.
\end{defi}
\bigskip

\subsection{ $q$-Laurent polynomial algebras}

Let ${\mathcal M}_q$ be a quantum matrix algebra as above and let
$\overline{\mathcal M}_q$ be the associated quasi-polynomial algebra. Let us for
simplicity denote the generators of ${\mathcal M}_q$ by $W_{\alpha j}$ and the
generators of $\overline{\mathcal M}_q$ by $w_{\alpha j}$,  where in both cases
$1\leq \alpha\leq n$ and $1\leq j\leq r$. 
\\
We first introduce the $nr\times nr$ matrix $H_{\mathcal M}= (H_{\alpha
j,\beta k})$  defined by
\begin{equation}\label{qpa}
w_{\alpha j}w_{\beta k}=q^{H_{\alpha j,\beta k}}w_{\beta k}w_{\alpha j} \, .
\end{equation} 
We are using the basis $\{w_{11},\dots,w_{1,r},w_{2,1},\dots,w_{2,r},\dots,w_{n,1},\dots,w_{n,r}\}$ and represent $H$ as an $n\times n$ block matrix consisting of $r\times r$ blocks $H_{\alpha \beta}$ (also see \eqref{qpa2}). 

\begin{defi}\label{def:H}
Let ${\mathcal L}={\mathcal L}_{\overline{\mathcal M}_q}$ be the
$q$-Laurent algebra generated by  $\overline{\mathcal M}_q$. We call
$H_{\mathcal M}$ the defining matrix of ${\mathcal L}$.
\end{defi} 
(More generally, we may consider a generic  algebra with generators
$w_{\alpha j}$ and relations given as in \eqref{qpa}). 
\medskip

The matrices $H$ will be examined in  \S \ref{se:Hmatrix} in many interesting cases. However, for many issues involving cluster algebras, it is much more useful to consider the  family  of $q$-commuting elements (quantum minors) $\chi_{\alpha j} \in{\mathcal M}_q$  which we  are now going to introduce.

To each $1\leq \alpha\leq n$ and $1\leq j\leq r$, let $\chi_{\alpha j}$ be the
quantum minor $\chi_\bold{j}^\balpha$ of biggest order $m$ in ${\mathcal M}_q$ fulfilling the following conditions: if $\gamma$
is a row number of $\chi_\bold{j}^\balpha$,  i.e. $\gamma \in \balpha$, then $\gamma\leq \alpha$ and if $c \in \bold{j}$ is a
column number, then $c\leq j$. Specifically,

\begin{enumerate}
\item for $\alpha\ge j$, $\chi_{\alpha j}:=\chi^{\{\alpha-j+1,
\alpha-j+2,\dots,\alpha\}}_{\{1,2,\dots,j\}}$,
\newline
\item for $\alpha<j$, $\chi_{\alpha
j}:=\chi^{\{1,2,\dots,\alpha\}}_{\{j-\alpha+1, j-\alpha+2,\dots,j\}}$.
\end{enumerate}
With reference to  \cite[\S6]{jz2}, this family of quantum minors corresponds
to the broken line $L^+$. The extreme opposite construction to the above, where
the conditions on the row and column numbers are changed to $\gamma\geq \alpha$
and $c\geq j$ corresponds to the broken line $L^-$.  Indeed one may define a
family for each broken line as defined in \cite{jz2}, but we will not pursue
this here. Notice, however, \S2.5. 

\medskip

\begin{defi}\label{family_minors}
We denote the family of minors $\chi_{\alpha j} \in {\mathcal M}_q$   given as above by ${\mathcal
  V}_{\mathcal M}^+$. 
\end{defi}

The family contains $nr$ elements. The following result follows from
\cite{jz2} {- where it was proved to hold for the FRT algebra $\mathcal{O}(M_q(n,r))$ - } {in combination with Section}~\ref{gen-fam}. See
\cite[Proposition~6.5]{jz2} for details.

\begin{prop}\label{propo}
Any two members $\chi_{\alpha j}$ and $\chi_{\beta k}$ of ${\mathcal
V}_{\mathcal M}^+$
$q$-commute. Thus a skew-symmetric integer matrix $\Lambda=\Lambda_{\mathcal
  M}$ may be defined by
\begin{equation}\label{rela}
\chi_{\alpha j}\chi_{\beta k}=q^{\Lambda_{\alpha j,\beta
k}}\chi_{\beta k}\chi_{\alpha j}.
\end{equation}
\end{prop}

As for the $q$-commutations between such minors,   each minor $\chi_{\alpha j}$
may be represented by its ``diagonal'' $\chi^d_{\alpha j}$ in
$\overline{\mathcal M}_q$. This is given by 
\begin{equation*}
\left\{\begin{array}{cl}\chi^d_{\alpha j}:=w_{\alpha ,j}w_{\alpha-1,j-1}\cdots w_{\alpha-j+1,1}&\text{if } \alpha\geq j 
\vspace{3pt}\\  \chi^d_{\alpha j}:=w_{\alpha,j}w_{\alpha-1,j-1}\cdots w_{1,j-\alpha+1}&\text{if }\alpha\leq j\end{array}\right.
\end{equation*}

\begin{prop}\label{propo-d}
Let $\Lambda=\Lambda_{\mathcal M}$ be as above. The following holds
in $\overline{\mathcal M}_q$:
\begin{equation}
\chi^d_{\alpha j}\chi^d_{\beta k}=q^{\Lambda_{\alpha j,\beta
k}}\chi^d_{\beta k}\chi^d_{\alpha j}.
\end{equation}
\end{prop}

\begin{proof} It is easy to see that there is a PBW type basis in an algebra
${\mathcal M}_q$ as above. We may choose this so that the monomials appearing
in the minors are elements of this basis. Furthermore, we may order the  monomials
according to a lexicographical ordering. The diagonal in any minor is biggest
among the monomials appearing as summands in it. Products of monomials may be
 expressed in the PBW basis using the relations of the algebra. When this is
done for the product of two minors, the highest order term will be a rewriting
of the product of the two diagonals, and here we may ignore
auxhilary terms of
lower order and rewrite according to ${\mathcal L}$. Thus, \eqref{rela} holds
on the level of the diagonals modulo lower order terms. To obtain the $q$
exponent ${\Lambda_{\alpha j,\beta
k}}$ it
thus suffices to  consider these
diagonals in ${\mathcal L}$. \end{proof}

\begin{prop}\label{211}
The matrix $\Lambda_{\mathcal M} $ defined in   \eqref{rela} is given by
\begin{equation}\label{THT}
 \Lambda_{\mathcal M}={\mathbb T}^tH_{\mathcal M}{\mathbb T},
\end{equation}
where ${\mathbb T}=({\mathbb
T}_{\beta k,\alpha j})$ is the upper-diagonal matrix  whose entries are either 1 or 0:
\begin{equation}\label{matrixTT}
{\mathbb T}_{\beta k,\alpha j}= \left\{
\begin{array}{ll}
1  & \mbox{ if } \exists \, x \in \{0,1,\dots,\min\{\alpha,j\}-1\} \mbox{ s.t. } (\beta,
k)=(\alpha -x,j-x), \vspace{3pt}
\\
0 & \mbox{ otherwise}
\end{array} \right. .
\end{equation}
\end{prop}
\begin{proof}
On the one hand we can write $T_{\beta k, \alpha j}= \sum_{x=0}^{\min\{\alpha,j\}-1} \delta_{\beta, \alpha -x} \delta_{k, j-x}$, so that
\begin{eqnarray}\label{THT-entries}
({\mathbb T}^tH_{\mathcal M}{\mathbb T})_{\alpha j, \beta k} &=& \sum_{x=0}^{\min\{\beta,k\}-1} {\mathbb T}^t_{\alpha j, \gamma i} H_{\gamma i, \beta -x \, k-x} = \sum_{\stackrel{x=0, \dots,\min\{\beta,k\}-1} {y=0, \dots, \min\{\alpha,j\}-1}} H_{\alpha -y \, j-y, \beta-x \, k-x}  \nn \\ 
&=&
\sum_{\stackrel{b=k, \dots,k-\min\{\beta,k\}+1} {a=j, \dots, j- \min\{\alpha,j\}+1}} H_{\alpha -j +a  \, a, \beta-  k+b \, b}.
\end{eqnarray}

On the other hand, following Proposition~\ref{propo-d}, we can compute the entries of the matrix  $\Lambda_{\mathcal M} $  from the commutation relations of the minors $\chi^d_{\alpha j}$. Let us start by considering the case
$\alpha \geq j$, $\beta \geq k$:
$$
\chi^d_{\alpha j} \chi^d_{\beta k} = \prod_{a =1}^j w_{\alpha-j+a,a} \prod_{b=1}^k w_{\beta-k+b,b} 
= q^{\Lambda_{\alpha j,\beta
k} } \chi^d_{\beta k} 
\chi^d_{\alpha j} ,
$$
where $\Lambda_{\alpha j,\beta
k} =\sum_{\stackrel{a=1,\dots , j}{b=1, \dots , k}} H_{(\alpha-j+a ) a,  (\beta-k+b)b} $. This
 coincides with $({\mathbb T}^tH_{\mathcal M}{\mathbb T})_{\alpha j, \beta k} $ as from \eqref{THT-entries} for $\alpha \geq j$, $\beta \geq k$. 
\\ The other cases are proved in analogous way, writing for $\beta <k$:
$$
\chi^d_{\beta k} = \prod_{b =1}^\beta w_{b, k - \beta+ b}= \prod_{c=k-\beta+1}^k w_{\beta-k+c,c} \, . 
$$

\end{proof}

\medskip

\subsection{Quantum seeds}\label{se:qs}
~
Quantum
seeds were introduced and first studied by Berenstein and Zelevinsky in
\cite{bz}. We will study a sub-class of quantum seeds: \begin{equation}
\{{\mathcal
  V}_{\mathcal M}^+,\Lambda_{\mathcal M} ,\tilde{B}_{\mathcal M} \}.
\end{equation}
(See \cite{bz}) for terminology).
\\
The elements $\chi_{\alpha j}$  - forming  the initial cluster ${\mathcal
  V}_{\mathcal M}^+$ (see Definition~\ref{family_minors}) -  satisfy the quasi-commutation relations \eqref{rela} that depend on the anti-symmetric matrix  $\Lambda_{\mathcal M}$.
According to the general theory, the adjacent clusters are obtained from the initial one via a process of mutations done in terms of the integer matrix $\tilde{B}_\mathcal{M}$.    The quasi-commutation matrix $\Lambda_{\mathcal M}$
and the exchange matrix $\tilde{B}_\mathcal{M}$ are required to satisfy a {\it compatibility condition} that ensures that the resulting clusters are still quasi-commuting.  
\\
In our case, 
the matrix $\tilde{B}_{\mathcal M}$ is a $nr\times c$ matrix, $c \leq nr$,
fulfilling the requirement that 
\begin{equation}\label{cc}
\Lambda_{\mathcal M}\tilde{B}_{\mathcal M}= \begin{pmatrix} -2{\mathbb
I}_c ~~\\ 0_{d\times c}\end{pmatrix} ,
\end{equation}
where $d=nr-c$. This equation implies that the pair $(\Lambda_{\mathcal M},\tilde{B}_{\mathcal M} )$ is {\it compatible} in the sense of \cite[Definition 3.1]{bz}. Indeed, it is a special instance of the compatibility condition, where generic diagonal matrices are  allowed in the place of $-2{\mathbb I}_c$.  
A natural, almost classical,  choice 
consists in requiring  that $\tilde{B}_{\mathcal M}$ is an $nr\times (nr-(n+r-1))$ matrix
corresponding to declaring the $n+r-1$ covariant minors
$\chi_{n1}\dots,\chi_{nr},\dots,\chi_{1r}$ as precisely the non-mutable elements.
However,  we want to maintain the freedom to choose a larger set of minors
as mutable, hence we solve \eqref{cc}
for $c=rk(\Lambda_{\mathcal M})$.  (See also \S \ref{se:cpairs}.)

\medskip

\subsection{Relations related to the determination of  the $B$ matrix}~

Let us consider as above  a set $\balpha=\{\alpha_1, \alpha_2, \dots, \alpha_m\}\subseteq\{1,\dots,n\}$ with $\alpha_1<\alpha_2< \dots< \alpha_m$.
Along with this, we introduce three  subsets $\balpha_L, \balpha_R$, and  $\balpha_o$ of $\balpha$: 
$\balpha_L=\{\alpha_2, \dots \alpha_m\}, \balpha_R=\{\alpha_1,  \dots ,\alpha_{r-1}\}$, and   $\balpha_o=\{\alpha_2,
\dots , \alpha_{r-1}\}$. We assume $m\geq2$. Then at most $\balpha_o$ may be empty. We
define analogous subsets for  $\bold{j}=\{j_1,j_2,\dots,j_m\}\subseteq\{1,\dots,r\}$, with $j_1<j_2<\dots <j_m$.
    
\begin{defi}In the above notation set\label{6}
\begin{equation}\begin{array}{lll}
X_t=\chi^{\balpha_L}_{\bold{j}_L}, \quad &X_b=\chi^{\balpha_R}_{\bold{j}_R}, \quad &X_o=
\chi^{\balpha_o}_{\bold{j}_o},\\
D=\chi^{\balpha}_{\bold{j}},&Y_L=\chi^{\balpha_R}_{\bold{j}_L},&Y_R=\chi^{\balpha_L}_{\bold{j}_R}.
\end{array}
\end{equation}
If $\balpha_o=\emptyset$ we set $X_o=1$.
\end{defi}

\medskip
\begin{prop}\label{213}The elements in Definition~\ref{6} $q$-commute. Moreover,
there are
integers $a^\balpha_\bold{j},c^\balpha_\bold{j}$ such that \begin{equation}
X_tX_b=q^{a^\balpha_\bold{j}}X_oD+q^{c^\balpha_\bold{j}}Y_L Y_R.
\end{equation} 
The elements in the $q$-Laurent algebra given by 
\begin{equation}
q^{a^\balpha_\bold{j}}X_oD(X_b)^{-1}\textrm{ and }\ q^{c^\balpha_\bold{j}}Y_L Y_R(X_b)^{-1}
\end{equation}
are invariant under the bar operation.
\end{prop}

\smallskip

\begin{proof}This result holds in ${\mathcal O}_q(M(n,r))$ (\cite[Corollary~6.14 and Theorem~6.17]{jz2}),   and the result
follows easily
from that. 
\end{proof}
\noindent
In the mentioned case in ${\mathcal O}_q(M(n,r))$, $a^\balpha_\bold{j}=0$, and $c^\balpha_\bold{j}=1$.

\medskip 

\noindent
The following is obvious:

\begin{prop}\label{214} The element 
$$X_o^{-1}D^{-1}Y_LY_R$$
commutes with all elements $R_\beta,C_k$.
\end{prop}

\medskip

\begin{prop}Introduce the integer $d=d_\bold{j}^\balpha$ such that $q^d
X_o^{-1}D^{-1}Y_LY_R$ is bar invariant. When we consider this as an element in
${\mathcal P}_q$ constructed from ${\mathcal M}_q$, different choices of
${\mathcal M}_q$  will yield the same element in ${\mathcal P}_q$.
\end{prop}

\begin{proof} The element in ${\mathcal O}_q(M(n,r))$ constructed according to
this recipe
has $d=0$. When we insert the elements $W_{\beta k}M_{\beta k}^{-1}$ in the positions
of the elements $Z_{\beta k}$ and move all $M_{\beta k}$ elements to the, say, right, we
get an element of the mentioned form for ${\mathcal M}_q$ possibly multiplied
with a monomial in the elements $R_\beta,C_k$. However, it follows easily from the
assumptions on the $M_{\beta k}$'s  that this monomial is a constant equal to 1. 
\end{proof}

\begin{cor}The constructions of a quantum seed for any broken line as given in
\cite{jz2} can be used for any quantum algebra ${\mathcal M}_q$ as above. 
\end{cor}

\begin{proof} One can observe that the element $X_o^{-1}D^{-1}Y_LY_R$ in
general satisfies the same kind of $q$-commutation relations as the specific
element considered in \cite{jz}. Specifically, it commutes with all elements of
the relevant set of variables with the exception of one (with which it
$q$-commutes with a non-zero $q$ exponent). With the crucial results
Proposition~\ref{213} and Proposition~\ref{214} at hand, the result now
follows by leafing through the arguments in \cite{jz} and observing that they
only rely on the formulas given in the mentioned propositions together
with the $q$-commutation property satisfied by the mentioned element.
\end{proof}
\medskip


\section{The inverse of the matrix $H$ of  the associated quasi-polynomial algebras}\label{se:Hmatrix}

Consider a quantized matrix algebra with $nr$ generators $w_{\alpha j}$ for $\alpha=1,\dots, n$ and $j=1,\dots, r$. Suppose we are given two integer $r\times r$ matrices $A=(a_{ij})$ and $M=(m_{ij})$  with $A^t=-A$ and that  the relations among
the generators $w_{\alpha j}$ are given as follows
\begin{eqnarray}\label{qapa}
 w_{\alpha j}w_{ \alpha k}&=&q^{a_{jk}}w_{ \alpha k}w_{\alpha j} ~,  \quad \forall  \alpha ,~ \forall j\leq k; \nn
\\ w_{\alpha j}w_{\beta k}&=&q^{m_{jk}}w_{\beta k}w_{\alpha j}~, \quad  \forall \alpha<\beta , ~\forall j,k;
\end{eqnarray}
\noindent
The above relations  can be rewritten (in accordance with \eqref{qpa}) as 
\be \label{qpa2}
w_{\alpha j}w_{\beta k}=q^{H_{\alpha j, \beta k}}w_{\beta k}w_{\alpha j}~, \quad   \alpha,\beta=1, \dots, n ; ~ j,k=1, \dots , r;
\ee
where the integers  $H_{\alpha j, \beta k} $ are the components of the matrix $H$ made up of $n\times n$ blocks $(H_{\alpha \beta})_{\alpha,\beta=1,\dots,n}$ defined in terms of the $r \times r$ matrices  $A,M$,
and $N=-M^t$ as
\begin{equation}\label{matrixH}
H=\left(\begin{array}{cccccc}A&M&M&M&\cdots&M\\N&A&M&M&\cdots&M\\N&N&A&M&\hdots&M\\N&N&N&A&\cdots&M	\\\vdots&\vdots&\vdots&\vdots&\cdots&\vdots\\N&N&N&N&\cdots&A
\end{array}\right)  .
\end{equation}
We will use both the notations $(H_{\alpha \beta})_{jk}$ and $H_{\alpha j, \beta k}$ to indicate the  component $(j,k)$ in the block $H_{\alpha \beta}$ of the  matrix $H$.
By construction, $H$ is  skew symmetric: $H_{\alpha \beta}= - (H_{\beta \alpha})^t$ (where here the transposition $^t$ indicates the transposition \textit{inside} the block).

\noindent
Our standing assumption will be that $A-N$ is invertible.
We set 
\begin{equation}
	X:=(A-N)^{-1}(A-M).
	\label{eq:}
\end{equation}
Notice that hence $X$ itself is invertible since $A-M=-(A-N)^t$.
 Finally,
we assume that $\I-X$ is invertible, or, equivalently, that $M-N$ is invertible.
\\

\subsection{First reductions}
We are interested in studying the invertibility of the matrix $H$. With this aim, 
we will now perform Gauss elimination on the blocks of $H$.
\noindent
Subtracting  (block) row 2 from 1, 3 from 2, etc. in $H$ results in the matrix 

\begin{equation*}
H_1=\left(\begin{array}{cccccc}A-N&M-A&0&0&\cdots&0\\0&A-N&M-A&0&\cdots&0\\0&0&A-N&M-A&\hdots&0\\\vdots&\vdots&\vdots&\vdots&\vdots&\vdots\\0&0&0&\cdots&A-N&M-A\\N&N&N&\cdots&N&A
\end{array}\right).
\end{equation*}

Thanks to our original assumptions, by using the first row we can remove the leftmost $N$ in the last row. Then we can use row 2 to remove the next $N$ in the last row, and so on  until an upper diagonal block matrix $H_2$ results:

\begin{equation*}H_2=\left(\begin{array}{cccccc}A-N&M-A&0&0&\cdots&0\\0&A-N&M-A&0&\cdots&0\\0&0&A-N&M-A&\hdots&0\\\vdots&\vdots&\vdots&\vdots&\cdots&\vdots\\0&0&0&\cdots&A-N&M-A\\0&0&0&\cdots&0&F
\end{array}\right),\end{equation*}
where 
\begin{equation*}F=(A-N)+N(\I+X+\cdots+X^{n-1}).
\end{equation*}
Using our assumptions, we easily find that 
\begin{equation}
	F=(M-NX^{n})(\I-X)^{-1}.
	\label{eq:F}
\end{equation}
If we furthermore assume that $M$ is invertible, we see that 

\begin{lemma}The null space of $F$ is equal to the +1 eigenspace of $(M^{-1}N)X^{n}$.
\end{lemma}

We now make further changes to the matrix $H_2$: we add $(A-M)(A-N)^{-1}$ times row $n-\alpha$ to row $n-\alpha-1$ for $\alpha=1,2,\dots, n-2$ (we multiply the blocks from the left). The resulting matrix is

\begin{equation}\label{H3}
H_3=\left(\begin{array}{cccccc}A-N&0&0&0&\cdots&(M-A)X^{n-2}\\0&A-N&0&0&\cdots&(M-A)X^{n-3}\\0&0&A-N&0&\hdots&(M-A) X^{n-4} \\\vdots&\vdots&\vdots&\vdots&\cdots&\vdots\\0&0&0&\cdots&A-N&M-A\\0&0&0&\cdots&0&F
\end{array}\right).\end{equation} ~
\bigskip

\noindent
In the following we will be interested in the following $r\times r$ matrices: 

\label{12}

\begin{equation*}{\small N_r=\left(\begin{array}{cccccc}-1&-1&-1&-1&\cdots&-1\\0&-1&-1&-1&\cdots&-1\\0&0&-1&-1&\hdots&-1 \\\vdots&\vdots&\vdots&\vdots&\cdots&\vdots\\0&0&0&\cdots&-1&-1\\0&0&0&\cdots&0&-1
\end{array}\right)} , \quad M_r= -N_r^t ,
 \end{equation*}
\begin{equation*}
S_r= {\small \left(\begin{array}{cccccc}0&1&0&0&\cdots&0\\0&0&1&0&\cdots&0\\0&0&0&1&\hdots&0 \\\vdots&\vdots&\vdots&\vdots&\cdots&\vdots\\0&0&0&\cdots&0&1\\-1&0&0&\cdots&0&0
\end{array}\right), \quad
X_r =\left(\begin{array}{cccccc}0&1&0&0&\hdots&0 \\0&0&1&0&\cdots&0\\0&0&0&1&\hdots&0\\\vdots&\vdots&\vdots&\vdots&\cdots&\vdots\\0&0&0&\hdots&\vdots&1
\\-1&-1&-1&-1&\cdots&-1\end{array}\right)} .
\end{equation*}
\normalsize

\noindent
It is easy to see that  the matrix $N_r$ is invertible, with inverse 
\begin{equation*}N_r^{-1}=\begin{pmatrix}-1&1&0&0&\hdots&0 \\0&-1&1&0&\cdots&0\\0&0&-1&1&\hdots&0\\\vdots&\vdots&\vdots&\vdots&\cdots&\vdots\\0&0&0&\hdots&\vdots&1
\\0&0&0&0&\cdots&-1\end{pmatrix}=- (M_r^{-1})^t\, \end{equation*}
and that 
$X_r=N_r^{-1}M_r$.   
Furthermore, observe that $X_r^{r+1}=\I$ (also see \eqref{Xl}),  $S_r^r=-\I$,  and  $S_r^{-1}=S_r^t$.

\medskip
 We shall take  special interest in the following cases corresponding to different choices for the matrices $A$ and $M$ which are building blocks of  the matrix $H$, introduced in \eqref{matrixH}. To distinguish among the different cases, we introduce a subscript for the relevant matrices. 
\begin{itemize}\label{cases}

\item \textsf{`Dipper-Donkin'}: 
$A_D=0$, $M_D=M_r$. Here  $N_D=N_r$,  $X$ becomes $X_D=X_r$ and $F$  becomes $F_D=N_DX_D(\I -X_D^{n-1})(\I-X_D)^{-1}$.  \vspace{5pt}
\item `\textsf{FRT'}: 
$A_S=-(M_r+N_r)$, $M_S=\I$. Here  $N_S=-\I$, $X$ becomes $X_{S}=S_r$ and $F$ becomes $F_{S}=(\I+S_r^n)(\I -S_r)^{-1}$. Notice that $A_S=S_r+S_r^2+\dots+S_r^{r-1}$ and $A-N$ becomes $ A_S+\I=2(\I-S_r)^{-1}$.
(We use the subscript $_S$ to refer to the FRT `standard'  quantum deformation of matrix algebras.)
\vspace{5pt}
\item `$\textrm{\textsf{Combined I}}$' :
 $A_{\co}=M_r+N_r$, $M_{\co}=M_r$. Here $N_{\co}=N_r$ and $X$ and $F$ become respectively $X_{\co}=X_D^{-1}$ and 
$F_{\co}=M_r(\I -X_{\co}^{n+1})(\I -X_{\co})^{-1}$.
\vspace{5pt}
\item `$\textrm{\textsf{Combined II}}$' :
 $A_{\cd}=M_r+N_r$, $M_{\cd}=N_r$. Here
$N_{\cd}=M_r$, $X$ becomes $X_{\cd}=X_D=X_r$ and $F_{\cd}=N_r(\I-X_r^{n+1})(\I-X_r)^{-1}$.

\end{itemize} \vspace{5pt}

With the above first two choices of matrices, the resulting matrices $H_D$ and $H_S$ describe the commutation relations  of the generators of the quasi-polynomial algebra associated respectively to the Dipper-Donkin (Definition~\ref{dddef}) and FRT \eqref{relations1} quantum matrix algebras. 
The last two choices are a `combination' of the previous ones (also see \S \ref{sect:full}).\\

\noindent
Thanks to \eqref{H3}, $det(H)= det(F) (det(A-N))^{n-1}$, thus  we easily obtain:

\begin{cor} \label{D-cor}In the above cases, the determinant of the matrix $H$ reduces to the following:
\begin{itemize}
\item $\det H_D=\det F_D$;
\item $\det H_S=2^{(r-1)(n-1)}\det F_S$;
\item $\det H_{\co}=\det F_{\co}$;
\item $\det H_{\cd}=(-1)^{r(n-1)}\det F_{\cd}$.
\end{itemize}
\end{cor}

\medskip

\subsection{The rank of $H$}
We determine the rank of the matrix $H$ in each of the cases listed above.
\subsubsection{Dipper-Donkin case}

The characteristic polynomial $p_D(z)$ of $X_D$ is easily computed to be
\begin{equation}\label{pD}
p_D(z)=\det(z \I  -X_D)=z^r+z^{r-1}+\dots+z+1=\prod_{p=1}^r(z-\varepsilon_p),
\end{equation}
where  $\varepsilon_p:=e^{2\pi i\cdot p/(r+1)}$,  $p=1,2,\dots,r$ are the $r$ distinct solutions to $\varepsilon^{r+1}=1,\varepsilon\neq1$.

It follows that the corank $c_D$ of $H_D$, which is equal to the corank of $F_D$ and so equal to the corank of $(\I-X_D^{n-1})$  can be determined as follows. From \eqref{pD}, it follows that $c_D$ is the number of integers $p=1,\dots, r$ for which $\frac{p(n-1)}{r+1}\in {\mathbb Z}$. We assume that $n>1$ and, with no loss of generality, we may assume that $n\geq r$. Let $s$ be the greatest common divisor of $r+1$ and $n-1$: 
$$
n-1=xs\textrm{ and }r+1=ys\textrm{ with $x,y$ relatively prime}.
$$  
\begin{prop}Let $s=g.c.d.(n-1,r+1)$, then the corank of $H_D$ is
$$c_D=\textrm{corank}(H_D)=s-1.$$
\end{prop}

\begin{proof} We keep using the notation introduced above.  A solution $p$ must satisfy $px=qy$ for some integer $q$. Hence since $x$ and $y$ are relatively prime, it has to be $p=p' y$ for some positive integer $p'$ and $p<r+1=ys$. Thus $p'=1,\cdots,s-1$ will yield the solutions, where the solutions $(p,q)$ are of the form $(jy,jx)$ for $j=1,\dots, s-1$. \end{proof}

\begin{cor}\label{cor:rkD}If $r=n$ then $$
	\textrm{corank}(H_D)=\left\{\begin{array}{l} 0\textrm{ if }n=r\textrm{ is even,}\\1\textrm{ if }n=r\textrm{ is odd}\, .\end{array}\right. 
$$
\end{cor}

\begin{proof} It is clear that $n-1$ and $n+1$ are relatively prime if $n$ is even. If $n$ is odd then $n-1$ and $n+1$ of course are even and have a common factor of $2$. \end{proof}

\medskip
\subsubsection{FRT case}\label{ssect:FRT}

 The matrix $X_S$ has characteristic polynomial $z^r+1$ and hence its eigenvalues are the $r$ $r$th roots of $-1$. One easily recovers the result from \cite{jj}: 

\begin{cor}\label{cnr}(cf. \cite[Prop. 4.5.]{jj}) Let $s=g.c.d.(n,r)$. Specifically, let $n=xs$ and $r=ys$. Then $H_S$ is non-invertible if and only if both $x$ and $y$ are odd. In this case,
\be
c_{n,r}:=\textrm{corank}(H_S)=s.
\ee
\end{cor}

\medskip
\subsubsection{`Combined' cases} 

Since $X_{\co}=X_D^{-1}$ and the
$r$-th roots of unity are invariant under taking inverses, the characteristic
polynomials for $X_D^{\pm1}$ are easily seen to be identical. Here, due to the
factor $(\I-X_r^{n+1})$ in $F_{\co}$, we let $s$ denote the greatest common factor of
$r+1$ and $n+1$:
$$
n+1=xs\textrm{ and }r+1=ys\textrm{ with $x,y$ relatively prime}.
$$  

The case $H_{\cd}$ is completely analogous.

\begin{prop}\label{coranks}The coranks are given by
$$\textrm{corank}(H_{\co})=s-1= {\textrm{corank} (H_{\cd})}.$$
\end{prop}

\begin{cor}
If $r=n$, then $\textrm{corank}(H_{\co})=r={\textrm{corank} (H_{\cd})}$.
\end{cor}
\noindent

\bigskip

\subsection{The inverse matrix $H^{-1}$} 
In this section we compute explicitly the inverse matrix of $H$ (when it exists).  We set $Y:=(A-M)(A-N)^{-1}$. Notice that $Y=(X^t)^{-1}$. For typesetting reasons, set $A_N:=(A-N)^{-1}$ and $X_{n-2}:=(\I+X+X^2+\cdots+X^{n-2})$. Consider the $n\times n$ block matrix $K$ made of $r\times r$ matrices given by \medskip
\be\label{K}
 K={\small
\left(\begin{array}{ccccccc}\I&Y-\I&Y^2-Y&Y^3-Y^2&\cdots&Y^{n-2}-Y^{n-3}&-Y^{n-2}\\
0&\I&Y-\I&Y^2-Y&\cdots&Y^{n-3}-Y^{n-4}&-Y^{n-3}\\
0&0&\I &Y-\I&\cdots&Y^{n-4}-Y^{n-5}&-Y^{n-4}\\
0&0&0&\I&\cdots&Y^{n-5}-Y^{n-6}&-Y^{n-5}\\\vdots&\vdots&\vdots&\vdots&\cdots&\vdots&\vdots\\0&0&0&0&\cdots&\I&-\I \\-NA_N&-NXA_N&-NX^2A_N&-NX^3A_N&\cdots&-NX^{n-2}A_N& \I+NX_{n-2}A_N
\end{array}\right).}
\ee
\smallskip
Then \begin{equation}H_3=KH\label{khh}\end{equation} (see \eqref{H3}), as one
can verify with some algebra by using immediate equalities 
like $Y(M-A)= (M-A)X$ and $X= A_{N}(A-M)$.

\medskip
\noindent
Since clearly $Y^\alpha=(A-N)X^\alpha (A-N)^{-1}$, we have $K=(A-N)K_2$, where
$$ {\small  K_2=
\left(\begin{array}{ccccccc}\I&X-\I &X^2-X&X^3-X^2&\cdots&X^{n-2}-X^{n-3}&-X^{n-2}\\
0&\I&X-\I &X^2-X&\cdots&X^{n-3}-X^{n-4}&-X^{n-3}\\0&0&\I&X-\I&\cdots&X^{n-4}-X^{n-5}&-X^{n-4}\\0&0&0&\I&\cdots&X^{n-5}-X^{n-6}&-X^{n-5}\\\vdots&\vdots&\vdots&\vdots&\cdots&\vdots&\vdots\\0&0&0&0&\cdots&\I&-\I \\-A_NN&-A_NNX&-A_NNX^2&-A_NNX^3&\cdots&-A_NNX^{n-2}& \I+A_NNX_{n-2}
\end{array}\right) (A_N).}
$$

\noindent Finally let \begin{equation}\label{H4} H_4=\left(\begin{array}{cccccc}\I&0&0&0&\cdots&-X^{n-1}\\0&\I&0&0&\cdots&-X^{n-2}\\0&0&\I&0&\hdots&- X^{n-3} \\\vdots&\vdots&\vdots&\vdots&\cdots&\vdots\\0&0&0&\cdots&\I&-X\\0&0&0&\cdots&0&A_NF
\end{array}\right).\end{equation}
A direct computation shows that  $K_2H=H_4$.

\medskip

\begin{lemma}
The matrix $H_4$ is invertible if and only if $F$ is, and in this case the inverse matrix is
\begin{equation}\label{H4inv}
H_4^{-1}= \begin{pmatrix}
\I & 0 & 0 & 0& \dots & X^{n-1} F^{-1} (A-N)
\\
0& \I &0 & 0& \dots & X^{n-2} F^{-1} (A-N)
\\
\vdots & & \vdots &&& \vdots
\\
0 &0 & 0 &   & \I&   X F^{-1} (A-N)
\\
0 &0 & 0 &   & 0&    F^{-1} (A-N)
\end{pmatrix}.
\end{equation}
\end{lemma}
\begin{proof}
By direct check one verifies that $H_4 H_4^{-1}= \I = H_4^{-1} H_4.$
\end{proof}

\begin{prop}If $H$ is invertible the blocks of its inverse are given as follows:
\begin{equation}\label{Hinv_gen}
H^{-1}_{\alpha \beta}=\left\{\begin{array}{cl}(\I-X^{n-\alpha}F^{-1}NX^{\alpha-1})(A-N)^{-1}&\textrm{ if }\alpha=\beta
 \vspace{3pt}
\\
(-X^{n-\alpha}F^{-1}NX^{\beta-1})(A-N)^{-1}&\textrm{ if }\alpha>\beta \vspace{3pt}\\ ~ - (H_{\beta \alpha}^{-1})^t &\textrm{ if }\alpha<\beta\end{array}\right. ,
\end{equation}
where $^t$ denotes the transposition inside the block.\\
In particular, for $H=H_D,n=r$ the inverse matrix has blocks
\begin{equation}\label{Hinv_DD}
(H_D^{-1})_{\alpha \beta}=\left\{\begin{array}{ll} -(\I + X_r^n(\I+X_r)^{-1}) N_r^{-1} &\textrm{ if
}\alpha=\beta \vspace{3pt}
\\
- X_r^{\beta-\alpha-1}(\I+X_r)^{-1} N_r^{-1}&\textrm{ if }\alpha>\beta  \vspace{3pt}
\\
- X_r^{\beta-\alpha+1}(\I+X_r)^{-1} N_r^{-1} = - ((H^{-1}_D)_{\beta \alpha})^t &\textrm{ if
}\alpha<\beta  \; ,\end{array} \right. 
\end{equation}
while for  $H=H_S, n=r+1$ 
\begin{equation}\label{Hinv_S} (H_S^{-1})_{\alpha \beta}=\left\{\begin{array}{ll}0&\textrm{ if
}\alpha=\beta\\ \frac12 (S_r^{\beta-\alpha+1}- S_r^{\beta-\alpha}) \quad &\textrm{ if }\alpha>\beta \vspace{3pt} \\  
\frac12 (S_r^{\beta-\alpha}- S_r^{\beta-\alpha-1})=
- ((H_S^{-1})_{\beta \alpha})^t&\textrm{ if }
\alpha<\beta \end{array} \qquad \right.  .
\end{equation}
\end{prop}
\medskip
\begin{proof}
When $H$ is invertible, then $H^{-1}= H_4^{-1} K_2$.  Thus, for $\alpha \neq n$ we compute
$$H^{-1}_{\alpha \alpha}= (K_2)_{\alpha \alpha}+ (H_4^{-1})_{\alpha n} (K_2)_{n \alpha}= 
(\I-X^{n-\alpha}F^{-1}NX^{\alpha -1})(A-N)^{-1}
$$
When $\alpha = n$ one computes $H^{-1}_{nn}= F^{-1} (A-N) (\I + (A-N)^{-1} N X_{n-2}) (A-N)^{-1}$, which is proved to coincide with  
$(\I-F^{-1}NX^{n-1})(A-N)^{-1}$ with some simple algebra. 
For $\alpha > \beta$, we easily compute $H^{-1}_{\alpha \beta}= - X^{n-\alpha} F^{-1}N X^{\beta -1} (A-N)^{-1}$  in both cases $\alpha \neq n$ and $\alpha = n$.  \\
In the Dipper-Donkin case $H=H_D$ (see page \pageref{cases}), $n=r$,  equation \eqref{Hinv_gen} reduces to \eqref{Hinv_DD}.
Indeed it is enough to observe that we here have 
$$\forall ~ \alpha:\ 
F^{-1} N X^\alpha = -X_r^{\alpha+1}(\I + X_r)^{-1}.
$$
 Finally, let us consider $H=H_S$ (full rank case). Recall that in this case, $X=S_r,~ N=-\I$, $(A-N)^{-1}= \frac{1}{2} (\I-S_r)$, and $F=(\I + S_r^n )(\I-S_r)^{-1}$. Equation \eqref{Hinv_gen} becomes
\begin{equation}\label{HS_inv}
(H_S^{-1})_{\alpha \beta}=\left\{\begin{array}{ll}
\half {(\I -S_r)(\I+S_r^{n-1})}{(\I + S_r^n)^{-1}} &\textrm{ if
}\alpha=\beta
 \vspace{3pt}
\\
\half S_r^{n+\beta-\alpha-1}{(\I-S_r)^2}{(\I+ S_r^{n})}^{-1}&\textrm{ if }\alpha>\beta  \vspace{3pt}
\\
- \half S_r^{\beta-\alpha-1}{(\I-S_r)^2}{(\I+ S_r^{n})}^{-1}=((- H_S^{-1})_{\beta \alpha})^t&\textrm{ if
}\alpha<\beta\end{array}\right. .
\end{equation}
For $n=r+1$ we get \eqref{Hinv_S}, using $S_r^n= -S_r.$ 
\end{proof}

\begin{rem} \label{rem_inv}

>From \eqref{Hinv_DD} we observe that in the case $n=r$, the diagonal blocks  $(H_D^{-1})_{\alpha \alpha}$ are independent of $\alpha$ and the off-diagonal blocks $(H_D^{-1})_{\alpha \beta}$  depend only on the difference $\beta-\alpha$, so that the block entries of $H_D^{-1}$  are constant along diagonals and are completely determined by the blocks in the, say, first block row. The same observation is valid for the matrix $H_S^{-1}$ ($r,n$ full rank case) as can be seen from \eqref{HS_inv}. 
\end{rem}

\medskip
\subsubsection{Dipper-Donkin, full rank case}
We compute explicitly the inverse matrix $H_D^{-1}$ in the case $n=r$ even.
\\

Let us initially assume that we are in the full rank case (but $r,n$ not necessarily equal).  
 We need to focus on the term $F_D$: we need the (integer) matrix
\begin{equation*}
\tilde F_D:=\frac{\I-X_D^{n-1}}{\I-X_D}= \I + X_D + \dots + X_D^{n-2}
\end{equation*}
to be invertible. Notice first that $X_D$ is diagonalizable with a diagonal $D$ consisting of the $r$ distinct points in $R_r={\mathbb C}\setminus\{1\}$ for which $\varepsilon_i^{r+1}=1$. It follows from the discussion before (see page \pageref{pD}) that for a pair $r,n$ corresponding to full rank, the map $\varepsilon_i\mapsto \varepsilon_i^{n-1}$ is a bijection of $R_r$ onto itself. Thus $F_D$ is similar to the diagonal matrix $(\I -D^{n-1})(\I -D)^{-1}$ of determinant $1$, and   we obtain

\begin{lemma}\label{D-reg}
$F_D$ is an integer matrix. If it is invertible, then it has  determinant 1.
\end{lemma}

\begin{lemma}
The following special cases hold:
\begin{itemize}
 \item If $(r,n)=(r,r+1)$ then $F_D=-N_D$.
\item If $(r,n)=(r,r)$ and $r$ is even, then $F_D=-N_DX_D^{-1}(\I+X_D)$. Furthermore, 
\be\label{inv_I+X}(\I+X_D)(\I+X_D^2+X_D^4+\cdots +X_D^{r})=\I \, . \ee
\end{itemize}
\end{lemma}
\begin{proof}
Small computations easily yield most of the above; the last identity follows since $\I+X_D+X_D^2+\cdots +X_D^r=0$. \end{proof}
\medskip
\noindent
Furthermore, notice that the interesting term in $F_D^{-1}$ is $\frac{\I-X_r}{\I-X_r^{n-1}}$. In the full rank case i.e. for g.c.d.(n-1,r+1)=1, choose $a,b$ $\in \IZ$ such that $a(n-1)+b(r+1)=1$. Without lost of generality we can assume that $a>0$ and that $a$ is the smallest positive  integer which satisfies the above equality. Then
\begin{equation}\label{def_a}
\frac{\I-X_r}{\I-X_r^{n-1}}=\I+X_r^{n-1}+\cdots+X_r^{(a-1)(n-1)}.
\end{equation}
For instance, when $n=r=2p$, then $a=p,~b=1-p$ and the above is simply the sum $\sum_{j=1}^p X_r^{j(r-1)}$. \\
\bigskip

Let $E_{ij}$ denote the matrix unit of position $i,j$ with $i,j=1,\dots,r$. (We will for convenience set $E_{ij}=0$ when at least one between $i$ and $j$ are not in $\{1,2,\dots,r\}$.) Define an $r\times r$ matrix $T=(t_{ij}) $ by  $T=\sum_{i=1}^{r-1} E_{i,i+1}$, i.e.
\be\label{matrixT}
 t_{ij}= \left\{
  \begin{array}{ll} 1 & \mbox{if }
  j-1 = i \, , j \ge 2 \\
 0 & \mbox{otherwise}.
\end{array} \right. \ee

\begin{lemma}For each $i=0,1,\dots, r$,
$$X_D^i=T^i+(T^t)^{r-i+1}-\sum_{s=1}^rE_{r-i+1,s}.$$
\end{lemma}

\begin{proof} This follows easily by (finite) induction.  \end{proof}

\noindent
For future use, also notice that the explicit form of $X_D^i$,   $i\leq r$, is 
\be\label{Xl}
X^i_D=
\left(
\begin{array}{ccc|c|ccc}
& &  & 0& &  & \\
& 0 & & \vdots &&\I_{r-i}  & \\
& &  & 0& &  & \\ \hline
-1& \dots &-1 &-1 &-1 & \dots & -1 \\ \hline
& &  & 0& &  & \\
& \I_{i-1} & & \vdots &&0  & \\
& &  & 0& &  &
\end{array}\right) \begin{array}{c}
\mbox{} \\
\mbox{} \\
\mbox{} \\
 \leftarrow \mbox{row~} r-i+1 \\
\mbox{} \\
\mbox{} \\
\mbox{} \\
\end{array}
\ee
and that we, furthermore,  by direct computation,  have 

\begin{lemma}For $i=0,1,\dots, r$,
\be\label{bigF}
X_D^i(\I-T)=T^i -T^{i+1}+(T^t)^{r-i+1}-(T^t)^{r-i}\ee
and 
\be\label{bigFminus}
X_D^{-i}(\I-T) = - \left( X^{i-1} (\I -T ) \right)^t.
\ee 
\end{lemma}

We can now compute explicitly the entries of the matrix $H_D^{-1}$ (also see  \eqref{Hinv4} below for a specific example: $n=r=4$):

\begin{prop}\label{prop_inv}
Let $n=r$ even. The components $(H_D^{-1})_{ab}$, $a,b=1, \dots, r^2$, of  the  inverse matrix of $H_D$ are constant along the diagonals, i.e. $(H_D^{-1})_{ab}= (H_D^{-1})_{a+c,b+c}$ for all admissible $c \in \IZ$, and they are given as follows: 
 $(H^{-1}_D)_{aa}=0$ and  for $b>a$ 
\begin{eqnarray}\label{HinvEXP}
(H^{-1}_D)_{ab}= -(H^{-1}_D)_{ba}=\left\{
\begin{array}{ll}
0 \quad & [b-a] = [1] \;  \\
1 \quad & [b-a]= [3],[5],\dots, [r+1]   \\
- 1 \quad & [b-a]= [2],[4], \dots, [r]
\end{array} \right.
\end{eqnarray}
where $[\cdot]$ denotes the class of an integer  in   $\IZ \slash (r+1)\IZ=  \{ [1],[2], \dots , [r+1] \}$. 
\end{prop}
\begin{proof}
  Starting with Remark \ref{rem_inv}, in order to determine the matrix $H_D^{-1}$ it is enough to determine the first row of blocks: $(H_D^{-1})_{1 \beta}$, $\beta=1, \dots , n= r$. Next 
$$
 (H_D^{-1})_{\alpha \beta} = \left\{ \begin{array}{ll}
(H_D^{-1})_{\alpha - \beta+1,1} =-((H_D^{-1})_{1,\alpha - \beta+1})^t & \alpha \geq \beta \vspace{5pt} \\ 
(H_D^{-1})_{1,\beta- \alpha+1} & \alpha \leq \beta .
\end{array} \right. 
$$

First we examine the diagonal block.   From \eqref{Hinv_DD},
$(H_D^{-1})_{11}=  (\I + X_D^r(\I + X_D)^{-1})(\I -T)$. By making use of equations \eqref{inv_I+X} and \eqref{bigF} we get
$$(H_D^{-1})_{11}= \left( -T + \sum_{i=1}^{r/2} \left( T^{2i -1} - T^{2 i}\right)  \right)+ 
\left( T^t + \sum_{i=1}^{r/2} \left( (T^t)^{r-2i +2} - (T^t)^{r - 2 i+1}\right)  \right),
$$
where the second term is the transposed of the first one. Since $T^{i}= \sum_{j=1}^{r-i} E_{j, j+i}$, we have that the block $(H_D^{-1})_{11}$ is completely determined by its (say) first row and (also using $T^r=0$) that this is given by
$
\left( (H_D^{-1})_{11} \right)_{1j}=(0,0,-1,1,-1,1, \dots, -1,1)
$ i.e. \be\label{bd}\left( (H_D^{-1})_{11} \right)_{1j}= \left\{ \begin{array}{ll}
0 & j=1,2\\
1 & j>1 \mbox{ odd} \\
-1 & j>2 \mbox{ even .}
\end{array} \right. \ee
An analogous argument shows that the off-diagonal blocks $(H_D^{-1})_{1 \beta}$, $\beta \neq 1$, 
consist of a sum of powers of $T$ and $T^t$. Specifically
\begin{eqnarray*}
(H_D^{-1})_{1 \beta}=T- T^2 +T^3 + \dots + T^{\beta -1}- T^{\beta+1}+ T^{\beta +2} + \dots + T^{r-2}-T^{r-1} -(T^t)^{r-1} + \nn
\\
(T^t)^{r-2}-(T^t)^{r-3}+ \dots +(T^t)^{r - \beta +2}-(T^t)^{r - \beta}+ (T^t)^{r - \beta -1}
+ \dots +T^t - \I  ~~~
\end{eqnarray*}
for $\beta$ even and 
\begin{eqnarray*}
(H_D^{-1})_{1 \beta}= \I -T+ T^2 + \dots + T^{\beta -1}- T^{\beta+1}+ T^{\beta +2} + \dots - T^{r-2}+T^{r-1}+  \nn
\\
  +(T^t)^{r-1}-(T^t)^{r-2}+ \dots +(T^t)^{r - \beta +2}-(T^t)^{r - \beta}+ (T^t)^{r - \beta -1}
+ \dots +(T^t)^2 - T^t  
\end{eqnarray*}
for $\beta$ odd.
Notice that contrary to the diagonal blocks,  the block $(H_D^{-1})_{1 \beta}$, $\beta \neq 1$, is not antisymmetric,   $(H_D^{-1})_{1 \beta} \neq - ((H_D^{-1})_{1 \beta})^t$, but still its entries are constant along the diagonals (due to the specific form of $T$), so that it is determined by its first row \textit{and} first column. From the above expressions we have
\be\label{bod}
\left( (H_D^{-1})_{1 \beta} \right)_{1j}= \left\{
\begin{array}{ll}
0 & \scriptstyle{ j~=~ \beta+1}
\\
1 & \scriptstyle{j~=~ \beta, ~\beta-2, ~\beta-4, \dots}
\\
-1 &  \scriptstyle{j~=~\beta-1, ~\beta-3, \dots}
\\
-1 & \scriptstyle{ j~=~\beta+2,~ \beta+4, \dots}
\\
1 & \scriptstyle{j~=~ \beta+3,~ \beta+5, \dots}
\end{array}
\right. \: ; \;
\left( (H_D^{-1})_{1 \beta} \right)_{i1}= \left\{
\begin{array}{ll}
0 &\scriptstyle{ i~=~ r- \beta+2}
\\
1 & \scriptstyle{i~= ~r- \beta,~ r-\beta-2,  \dots}
\\
-1 &\scriptstyle{ i~= ~r- \beta+1,~r- \beta-1, \dots}
\\
1 & \scriptstyle{i~= ~r-\beta+3,~ r- \beta+5, \dots}
\\
-1 & \scriptstyle{i~=~r- \beta+4, ~r-\beta+6, \dots}
\end{array}
\right.
\ee
for both $\beta$ odd or even ($\beta \neq 1$), $i,j=1, \dots , r$.
\\

We can now prove that the components $(H_D^{-1})_{ab}$, $a,b=1, \dots , r^2$, are constant along the diagonals. It remains  to prove it when passing from a block to a nearby one, for instance we need to prove that $((H_D^{-1})_{1 \beta})_{ir}= ((H_D^{-1})_{1, \beta+1})_{i+1,1}$ for all $i\neq r$, $\beta \neq 1$.  To do that it is enough to compare the expressions of   $((H_D^{-1})_{1 \beta})_{ir}=
((H_D^{-1})_{1 \beta})_{1, r-i+1}$  and $ ((H_D^{-1})_{1, \beta+1})_{i+1,1}$ resulting from \eqref{bod}.
Similarly we can prove the result for all other cases:
\begin{eqnarray*}
&&((H_D^{-1})_{1 2})_{ri}= ((H_D^{-1})_{22})_{1,i+1} \quad , \quad 
((H_D^{-1})_{1 1})_{ir}= ((H_D^{-1})_{12})_{i+1,1}  \quad ,
\\
&&((H_D^{-1})_{\alpha \beta})_{ri}= ((H_D^{-1})_{\alpha+1, \beta})_{1,i+1}, ~~\alpha+1<\beta,~~i \neq r ~.
\end{eqnarray*}
 Summarizing, we conclude that $H_D^{-1}$ is determined by its first row. From the previous computations,
$$ 
\begin{array}{ccc}
(H_D^{-1})_{1 b}= ( \dots &\underbrace{|\dots, 1,-1,1,-1,1,0,-1,1,-1,1 , \dots|} &\dots)\\
&\uparrow&
\\ & \scriptstyle{~ block ~ (H_D)_{1\beta}  }&
\end{array},
$$
where the $0$'s in  block $\beta$  occurs in position $b=(\beta-1)(r+1)+2$, the $1$'s for
$b= \dots, (\beta-1)(r+1)-1,(\beta-1)(r+1)+1, (\beta-1)(r+1)+4, (\beta-1)(r+1)+6, \dots$, and the $-1$'s in the remaining positions.
>From this, together with the equation $(H_D^{-1})_{ab}= (H_D^{-1})_{a+c,b+c}$, we deduce formula \eqref{HinvEXP}.
\end{proof}

\medskip

\subsubsection{Dipper-Donkin case $r=n$ odd}

In this last part of the section we return to the result of Corollary~\ref{cor:rkD}. For $r=n$ odd the matrix $H_D$ is not invertible, nevertheless we can construct a `partial left inverse' as follows. {(See  also \S \ref{se:cpairs}.)}

Let us introduce the invertible matrix $U_n=\I+\sum_{1\leq
  2i+1<n}E_{n,2i+1}$. We further introduce the matrix
$P:=\sum_{k=1}^nE_{n,k}$ which has all entries zero, but for the last row of 1's.  Thus we can rewrite the matrices $X_r$ and $T$ as $X_r=N_r^{-1}+\I-P$  and $T=X_r+P$. The following is elementary:
\begin{equation*}
  U_n(\I+X_r)=\left\{\begin{array}{ll}\I+(X_r+P)=2+N_r^{-1}&\textrm{ for }n\textrm{
even}\\
\I+(X_r+P)-E_{n,n}&\textrm{ for }n\textrm{ odd}.
  \end{array}\right.  
\end{equation*}
Further, we have easily  that the term $A_N F$ in the matrix \eqref{H4}  here is given as  $A_NF=(\I+X_r^{-1})$.

Thus the invertible integer matrix $V_n:=(\I+T)^{-1}U_n X_r$ 
of
determinant $1$ is such that $V_n(\I+X^{-1})=\I$ if $n$ is even and
$V_n(\I+X^{-1})=\I-E_{n,n}+E_{n-1,n}-E_{n-2,n}+\dots-E_{1,n}$ if $n$ is odd.
For $n$ odd, we introduce the notation 
\be\label{Cn}
\tilde{E}_n:=E_{n,n}-E_{n-1,n}+E_{n-2,n}+\dots+E_{1,n} \, .\ee 

\noindent
Consider the matrix 

\begin{equation*}Z_n=\left(\begin{array}{cccccc}
\I&0&0&\dots&0&X^{n-1}V_n\\
0&\I&0&\dots&0&X^{n-2}V_n\\
0&0&\I&\dots&0&X^{n-3}V_n\\
\ddots&\ddots&\ddots&\dots&\ddots&\ddots\\
0&0&0&0&\I&XV_n\\
0&0&0&0&0&V_n
\end{array}\right).\end{equation*}
\noindent
Then by using the matrix $H_4=K_2H$ in  \eqref{H4}, we have  
$
Z_n K_2H =\I_{n} $ when $n$ is even and

\be\label{ZKH}
Z_n K_2H=\left(\begin{array}{cccccc}
\I&0&0&\dots&0&-X^{n-1}\tilde{E}_n\\
0&\I&0&\dots&0&-X^{n-2}\tilde{E}_n\\
0&0&\I&\dots&0&-X^{n-3}\tilde{E}_n\\
\ddots&\ddots&\ddots&\dots&\ddots&\ddots\\
0&0&0&0&\I&-X \tilde{E}_n\\
0&0&0&0&0&\I-\tilde{E}_n
\end{array}\right) 
\ee
when $n$ is odd,
in agreement  with Corollary~\ref{cor:rkD}. (For $n=r$ even, $Z_n= H_4^{-1}$, see \eqref{H4inv}.) 

\medskip
\subsubsection{The Dipper-Donkin case for $n=r+1$}

We finally analyse  the inverse matrix $H_D^{-1}$ for $n=r+1$.  In this case \eqref{Hinv_gen}
simplifies considerably because in this case $F=-N$.
Then
\begin{equation}\label{Hinv_DD_n=r+1}
(H_D^{-1})_{\alpha \beta}=\left\{\begin{array}{ll} (\I + X_r^{-1})(\I -T) &\textrm{ if
}\alpha=\beta \vspace{3pt}
\\
 X_r^{\beta-\alpha-1}(\I-T) &\textrm{ if }\alpha>\beta
\\
 - ((H^{-1}_D)_{\beta \alpha})^t &\textrm{ if
}\alpha<\beta  \; ,\end{array} \right. 
\end{equation}
so that $H_D^{-1}$ is constant along (block) diagonals as it was for $n=r$.\\ 
Next, from \eqref{bigF} it is immediate to see that 
$$
(H_D^{-1})_{\alpha \alpha}= T^t-T= \sum_{i=1}^{r - 1} ( E_{i+1,i}- E_{i,i+1} )
$$ 

\noindent
We need only to compute the blocks  $(H_D^{-1})_{\alpha \beta}$  for (say) $\alpha=1$, $\beta \geq 2$. 
First, for $1 \leq \alpha < \beta$, from \eqref{bigFminus} together with \eqref{bigF} we get
$$
H^{-1}_{\alpha \beta}= -(H^{-1}_{\beta \alpha})^t= T^{\beta-\alpha}-T^{\beta -\alpha +1} +
(T^t)^{r-\beta+\alpha+1}-(T^t)^{r-\beta +\alpha } \, ,$$ 
so that

\begin{lemma}
For each $\beta= 2, \dots r$ it holds that
\begin{eqnarray} \label{A}
(H_D^{-1})_{1 \beta}&=&
T^{\beta-1}-T^{\beta} +
(T^t)^{r-\beta+2}-(T^t)^{r-\beta +1 } \nn
\\& = &\sum_{i=1}^{r - \beta+1} E_{i,i+\beta-1}  -
\sum_{i=1}^{r - \beta} E_{i,i+\beta} +  \sum_{i=1}^{\beta-2} E_{r-\beta +i+2,i}-  \sum_{i=1}^{\beta-1} E_{r-\beta +i+1,i}
\end{eqnarray}
\end{lemma}
\noindent
Moreover, a little algebra shows that:
\begin{lemma}\label{DDn=r+1}
For $n=r+1$, the entries of the matrix $H_D^{-1}$ are constant along the diagonals. Specifically,   for each $a,b=1, \dots, r(r+1)$: 
$(H_D^{-1})_{ab}= (H_D^{-1})_{a+c,b+c}$ for all admissible $c \in \IZ$.
\end{lemma}

\medskip
\subsubsection{The FRT case for $n=r+1$}~
Let us now address  the FRT case as given on page \pageref{cases}.

\begin{lemma}
The following special cases hold: 
\begin{itemize}
 \item If $(r,n)=(r,r+1)$ then $F_S=\I$.
\item If $(r,n)=(r,r-1)$ then  $F_S=-S_r^{-1}$
\item If  $(r,n)=(r,2r)$ then $F_S=2 (\I-S_r)^{-1}$.
\end{itemize}
\end{lemma}

We analyse from \eqref{Hinv_S} the inverse matrix $H_S^{-1}$ for $n=r+1$. As observed in Remark 
\ref{rem_inv},  it is enough to compute the blocks $(H_S^{-1})_{\alpha \beta}$  for (say) $\alpha=1$, $\beta \geq 2$.

\begin{prop}
For each $\beta= 2, \dots r$ it holds that
\be\label{AA}
2 (H_S^{-1})_{1 \beta}= \sum_{i=1}^{r - \beta+1} E_{i,i+\beta-1} -  \sum_{i=1}^{\beta-1} E_{r-\beta +i+1,i} -
\sum_{i=1}^{r - \beta+2} E_{i,i+\beta-2} +  \sum_{i=1}^{\beta-2} E_{r-\beta +i+2,i}
\ee
\end{prop}
\begin{proof}
By induction one can prove that
for each $ \ell \leq r$, 
\be\label{Spower}
S^\ell = \sum_{i=1}^{r - \ell} E_{i,i+\ell} -  \sum_{i=1}^{\ell} E_{r-\ell +i,i}= T^\ell -(T^t)^{r-\ell} \; ,
\ee
and the formula  \eqref{AA} then follows immediately from \eqref{Hinv_S}.
\end{proof}
Notice that the matrix $H_S^{-1}$ does not have constant values along the diagonals, contrary to what happens for the Dipper-Donkin cases $n=r$  (see Prop. \ref{prop_inv}) and  $n=r+1$ (Lemma \ref{DDn=r+1}). For example, for $n=r+1=6$ one has
$$
 (H_S^{-1})_{12 }|(H_S^{-1})_{13}= \half
\left(\begin{array}{ccccc|ccccc}
-1 & 1 & 0 & 0 & 0 &0 &-1&1&0&0
\\
0&-1&1&0&0&0&0&-1&1&0
\\
0&0&-1&1&0&0&0&0&-1&1
\\
0&0&0&-1&1&-1&0&0&0 &-1
\\
-1&0&0&0&-1&1&-1&0&0&0
\end{array}\right) .
$$  
\bigskip


\section{Block diagonals. Degrees}\label{se:degrees}

Recall that a skew symmetric $N\times N$ integer matrix $J$ of $\textrm{corank}(J)=c$, when viewed as a quadratic form, can be transformed to a block diagonal form by an integer matrix $L$ of determinant 1. Specifically, $$L^t JL=\textrm{Diag}(D_1,D_2,\dots, D_k,0,0,\dots,0)$$
where $k=\frac12(N-c)$ and each $D_i$, $i=1,\dots,k$ is a non-trivial skew symmetric $2\times2$ integer matrix.  
\\

\begin{rem}
By the work of De Concini and Procesi \cite{pdc} the block diagonal form yields
the degree of the quantized matrix algebra in case $q$ is a primitive $m$th root
of unity.
\end{rem}

We can apply this result to $J=H_D$ (see page \pageref{cases}). First, let us
assume that $H_D$ is invertible. 

\begin{cor}Suppose $H_D$ is invertible. Then a block diagonal
form of $H_D$ consists of $\frac12 rn$ blocks of the form
$\left(\begin{array}{cc}0&1\\-1&0\end{array}\right)$.
\end{cor}

\begin{proof} It follows, by combining Lemma~\ref{D-reg} with Corollary~\ref{D-cor}, that $\det H_D=1$ in this case. 
\end{proof}

As for the situation for FRT, we have the
following result adapted to the current terminology. Let
$d_0=\lfloor\frac{n+r-1}2\rfloor$.

\begin{prop}\label{jj22}\cite[Proposition~4.11]{jj}
  The non-trivial blocks in a block diagonal form of the defining
  matrix $H_S$ are: $d_0$ matrices of the form
  $\left(\begin{array}{cc}0&1\\-1&0\end{array}\right)$ and 
  $\max\{0, \frac{nr-c_{n,r}}{2} - d_0\}$ matrices of the form
  $\left(\begin{array}{cc}0&2\\-2&0\end{array}\right)$ or 
  $\left(\begin{array}{cc}0&4\\-4&0\end{array}\right)$, where $c_{n,r}=\textrm{corank}(H_S)$.
\end{prop}

Recall: $n=xs$ and $r=ys$. According to Corollary~\ref{cnr}, $c_{n,r}=s$, but only in case both $x$ and $y$ are odd. We assume throughout that $r\geq
n>1$. In case of a regular matrix $x,y$ must have opposite parities. According to 
 {Corollary~\ref{D-cor}}, the determinant of $D$ is given by  
$$
\det D=2^{(n-1)(r-1)}\det F_S
$$ 
 where $F_S=\frac{1+S^n}{1-S}$.  In this case it is easy to see that $\frac12 nr-d_0\geq0$. If $f$ denotes the number
of blocks with $4$'s then it is easy to see that Proposition~\ref{jj22}
yields$${\det H_S}=\left\{\begin{array}{l}2^{nr -n-r+1+2f}\textrm{ if
}s\textrm{ is odd} \\2^{nr
-n-r+2+2f}\textrm{ if }s\textrm{ is even}.\end{array}\right. $$ The number $f$
was only determined in a few special cases in \cite{jj}. We can now use
Corollary~\ref{D-cor} to determine it. Specifically,  one may
use some elementary Gauss Elimination moves on $1+S^n$, or $F$, to conclude that
$\det(1+S^n)=2^s$, or equivalently,  $\det F=2^{s-1}$. We give a short sketch of  this result. Let us introduce a more general configuration $1+\varepsilon_1 T^n+\varepsilon_2 (T^t)^{r-n}$, where $\varepsilon_1,\varepsilon_2$ are $\pm1$. It may be assumed that they are never  $-1$ at the same time. $T$ is the $r\times r$ matrix of (\ref{matrixT}). We are interested in $\varepsilon_1=-\varepsilon_2=1$, but will encounter more general configurations in the reduction process.  It is indeed easy to see that we can reduce the dimensions by splitting off diagonals of $1$'s as follows:
\begin{eqnarray*}
r>2n: (n,r,\varepsilon_1,\varepsilon_2)&\rightarrow&(n,r-n,\tilde\varepsilon_1,\tilde\varepsilon_2)\\r<2n: (n,r,\varepsilon_1,\varepsilon_2)&\rightarrow&(2n-r,n,\hat\varepsilon_1,\hat\varepsilon_2).
\end{eqnarray*}
As an example, consider, for $r>2n$, the configuration $(n,r,1,-1)$: Adding the top $n$ rows to the bottom $n$ rows gives a matrix in which the first $n$ rows carry so-called leading $1$'s and this part can then be ignored. The remaining $(r-n)\times (r-n)$  matrix then evidently has signs $\varepsilon_1=1=\varepsilon_2$.

In case $2n=r$ we are done in one step, and here  we obtain the lower part of the diagonal consisting of $s$ places with the value $1-\varepsilon_1\varepsilon_2$. In case $\varepsilon_1\varepsilon_2=1$ the original matrix is thus singular. Observe that all matrices can be viewed as being built up of $s\times s$ blocks of $\pm{\mathbb I}_s$ which means that we may as well set $s=1$ in the reduction process, returning  it to its original value only in the end. We are considering the case where the matrix is regular, and since $x,y$ are relatively prime, and have opposite parities, the situation $y=2x$ implies $x=1$. As before, if we are in this situation, we are done in one step. If we are not in this case we will eventually get there according to the above strategy. 
\\
In summary, we have then obtained

\begin{prop}Let $f$ denote the number of times a block
$\left(\begin{array}{cc}0&4\\-4&0\end{array}\right)$ appears in the block
diagonal of $H_S$ when the latter is regular. Then  
$$f=\lfloor\frac{s-1}2\rfloor.$$
\end{prop}

\setcounter{section}{4}
\subsection{Concerning  the block diagonal form of a non-regular $H$ }
To deal with the non-regular cases, we make the following observations.

\medskip
\noindent
Suppose that there exist an integer matrix $G$ of determinant 1 such that
$GHG^t=\mathbb D$:
 
\begin{equation}\left(\begin{array}{cc}\label{41}
g_1&g_2\\ g_3&g_4\end{array}\right)\left(\begin{array}{cc}
h_1&h_2\\ -h_2^t&h_3\end{array}\right)\left(\begin{array}{cc}
g_1^t&g_3^t\\ g_2^t&g_4^t\end{array}\right)={\mathbb D}=\left(\begin{array}{cc}
D&0\\ 0&0\end{array}\right),\end{equation}
where we have split the matrices in blocks (of suitable dimensions) and $D$ is a non-degenerate block diagonal matrix.
This implies

\begin{equation}\left(\begin{array}{cc}
(g_1h_1-g_2h_2^t)g_1^t+(g_1h_2+g_2h_3)g_2^t&(g_1h_1-g_2h_2^t)g_3^t+(g_1h_2+g_2h_
3)g_4^t\\
(g_3h_1-g_4h_2^t)g_1^t+(g_3h_2+g_4h_3)g_2^t&
(g_3h_1-g_4h_2^t)g_3^t+(g_3h_2+g_4h_3)g_4^t\end{array}\right)\label{corner1}
=\left(\begin{array}{cc}
D&0\\ 0&0\end{array}\right).\end{equation}

Notice that we are not assuming that the blocks are of the same size. In applications below, this is far from being the case.

Suppose furthermore that there exist a matrix $Z$, similarly decomposed into blocks, such that
\begin{equation}
Z\left(\begin{array}{cc}
h_1&h_2\\ -h_2^t&h_3\end{array}\right)=\left(\begin{array}{cc}
z_1&z_2\\ z_3&z_4\end{array}\right)\left(\begin{array}{cc}
h_1&h_2\\ -h_2^t&h_3\end{array}\right)=\left(\begin{array}{cc}
z_1h_1-z_2h_2^t&z_1h_2+z_2h_3\\
z_3h_1-z_4h_2^t&z_3h_2+z_4h_3\end{array}\right)=\left(\begin{array}{cc}
L&t_1\\ 0&t_2\end{array}\right)\label{corner2}.
\end{equation}

At the moment we just assume that $L$ is a general matrix. Assume furthermore that
$t_2=0$ and that $z_4$ is invertible. Set
$x=z_4^{-1}z_3$. It follows that $h_2^t=xh_1$, $h_2=-h_1x^t$, and
$h_3=xh_1x^t$. Then the upper left hand corner of (\ref{corner1}) gives:
\begin{equation}
  (g_1-g_2x)h_1(g_1-g_2x)^t=D.
\end{equation}
Similarly, we get from (\ref{corner2}) that
\begin{equation}(z_1-z_2 x)h_1=L\label{M-eq}\end{equation}
More generally, we get 
\begin{equation}
  \left(\begin{array}{cc}
     
(g_1-g_2x)h_1(g_1-g_2x)^t&(g_1-g_2x)h_1(g_3-g_4x)^t\\-(g_3-g_4x)h_1(g_1-g_2x)^t&
(g_ 3-g_4x)h_1(g_3-g_4x)^t\end{array}\right) =\left(\begin{array}{cc}
      D&0\\ 0&0\end{array}\right).\end{equation}
Since $(g_1-g_2x)$ and $h_1$ are invertible it follows that
\begin{equation}
g_3-g_4x=0.
\end{equation} 

Now  observe:

\begin{lemma}\label{lem-D}If $x$ is an integer matrix, then $D$ is a block diagonal
form of $h_1$.
\end{lemma} 
\begin{proof}Notice that in this case,  $\left(\begin{array}{cc}
\varepsilon_1\I&0\\ \varepsilon_2 x&\varepsilon_3\I\end{array}\right)$, for suitable choices of signs $\varepsilon_1,\varepsilon_2$, and $\varepsilon_3$, is an integer matrix of determinant 1 for which 
\begin{equation}
\label{x-int}  \left(\begin{array}{cc}
\varepsilon_1\I&0\\ \varepsilon_2 x&\varepsilon_3\I\end{array}\right)\left(\begin{array}{cc}
h_1&0\\ 0&0\end{array}\right)\left(\begin{array}{cc}
\varepsilon_1\I&0\\ \varepsilon_2 x&\varepsilon_3\I\end{array}\right)^t= \left(\begin{array}{cc}
h_1&h_2\\ -h_2^t&h_3\end{array}\right).
\end{equation}
Inserting this into (\ref{41}) gives that $\left(\begin{array}{cc}
D&0\\ 0&0\end{array}\right)$ is a block diagonal form of $\left(\begin{array}{cc}
h_1&0\\ 0&0\end{array}\right)$.
\end{proof}

\noindent
Observe that 

\begin{equation}\label{49}
\left(\begin{array}{cc}
z_1&z_2\\ z_3&z_4\end{array}\right)=\left(\begin{array}{cc}
z_1-z_2x&z_2\\ 0&z_4\end{array}\right)\left(\begin{array}{cc}
\I&0\\ z_4^{-1}z_3&\I\end{array}\right)
\end{equation}
so that, if $\det Z=1$, then $\det(z_1-z_2x)\det z_4=1$. From Lemma~\ref{lem-D} and equations  \eqref{M-eq}, \eqref{49} we then conclude
\begin{prop}
If $x$ is an integer matrix and if $\det Z=1$, then 
$$\det D=\det L\det z_4.$$
\end{prop}

\subsubsection{The block diagonal form of a non-regular $H_S$ }

Consider the singular case for FRT. As in \S \ref{ssect:FRT}, let $r=ys$, where
$s=g.c.d.(n,r)$. Let
$$G_S=\left(\begin{array}{ccccccc}
\I&0&0&0&\dots&0&0\\0&\I&0&0&\dots&0&0\\0&0&\I&0&\dots&0&0\\0&0&0&\I&\dots&0&0\\
\vdots&\vdots&\vdots&\vdots&\vdots&\vdots&\vdots\\0&0&0&0&\dots&\I&0\\\I&-\I&\I&
-\I&\dots&-\I&\I\end{array}\right)$$
be an $y\times y$ block matrix where each block is an $s\times s$ matrix so
that $\I={\mathbb I}_s$. 
Similarly, let $$\widehat{G_S}=\left(\begin{array}{ccccccc}
\I&0&0&0&\dots&0&0\\0&\I&0&0&\dots&0&0\\0&0&\I&0&\dots&0&0\\0&0&0&\I&\dots&0&0\\
\vdots&\vdots&\vdots&\vdots&\vdots&\vdots&\vdots\\0&0&0&0&\dots&\I&0\\0&0&0&
0&\dots&0&G_S\end{array}\right)
$$
be an $n\times n$ block matrix of $r\times r$
blocks. We are interested in stydying the effect of a multiplication from the
left on the two sides of (\ref{khh}) (case of $H=H_S$) by $\widehat{G_S}$. 

We first investigate $\widehat{G_S}H_3$. This will be the right hand side of (\ref{corner2}). It may be seen that $G_S(\I+S^n)$ is a
matrix whose bottom $s$ rows
are zeros. 
 A similar statement holds for $G_S(\I+S^n)(\I-S)^{-1}$. Let $\widehat L$ denote the
$(r-s)\times (r-s)$ matrix obtained from $(\I+S^n)(\I-S)^{-1}$ by removing the
last $s$ columns and last $s$ rows. In the terminology of (\ref{corner2}) we then have
\begin{equation}L=\left(\begin{array}{cccccc}A-N&0&0&\dots&0&0\\0&A-N&0&\dots&0&0\\0&0&A-N&\dots&0&0\\
\vdots&\vdots&\vdots&\vdots&\vdots&\vdots\\0&0&0&\dots&A-N&0\\0&0&
0&\dots&0&\widehat L\end{array}\right).\end{equation}The blocks $A-N$, of which there are $n-1$, are of size $r\times r$. It may be seen that $\widehat L$ is an integer
matrix of determinant 1. Since this computation is
very analogous, indeed almost identical, to a case for the Dipper-Donkin algebra treated below, we omit
the details. {This was the only unknown piece of $\det L$.}

Now to $\widehat{G_S}K$: This will be the matrix $Z$ of
the previous considerations. Specifically, the matrix $z_3$ will be an
$s\times(nr-s)$ matrix which, together with $z_4$ make up the bottom $s$ rows
of $Z$. The effect of the multiplication is, basically, that we multiply the
bottom (block) row of $K$ (\ref{K}) by $G_S$, leaving everything else unchanged.
We have that ${\mathbb I}+NX_{n-2}A_N=\frac12(\I+S^{n-1})$. In this matrix, we
are in particular interested in the bottom row of $s\times s$ blocks and thus
write

\begin{equation}\label{bottom s}\frac12 G_S(\I+S^{n-1})=\frac12
G_S(\I-S^{-1})+\frac12 G_S(\I+S^n)S^{-1}.\end{equation}

Notice that the matrices $T$ and $S$ can be defined in any positive dimension
$d\in{\mathbb N}$. Specifically,  $T=\sum_{b=1}^{d-1}E_{b,b+1}$ and
$S=T-E_{d,1}.$  We will denote these matrix by $T_d$ and $S_d$ to clarify the
notation in what comes. The last
term in (\ref{bottom s}) may be ignored for our present purposes. We have that
$\I_r-S_r^{-1}=\I_r-T^t_r+E_{1,r}$. We write this in
terms of $s\times s$ blocks as follows:
\begin{equation}
\I_r-S^{-1}_r=\left(\begin{array}{ccccccc}
\I_s-T^t_s&0&0&0&\dots&0&E_{1,s}\\-E_{1,s}&\I_s-T^t_s&0&0&\dots&0&0\\0&-E_{1,s}&\I_s-T^t_s&0&\dots&0&0\\0&0&-E_{1,s}&\I_s-T^t_s&\dots&0&0\\
\vdots&\vdots&\vdots&\vdots&\vdots&\vdots&\vdots\\ 0&0&0&0&\dots&\I_s-T^t_s&0\\0&-0&0&0&\dots&-E_{1,s}&\I_s-T^t_s\end{array}\right).
\end{equation}

We then easily obtain: Set $R_s=\frac12(\I-S^{-1}_s)$. This
is an $s\times s$ matrix, and the last row of $\frac12 G_S(\I+S^{n-1})$ is given
as (reading from left to
right,  separated by vertical lines $\vert$): $( R_s\vert -R_s\vert R_s\vert \dots\vert
-R_s\vert R_s)$. To connect with the previous, observe that $z_4$ here is represented by $R_s$. Thus,  $\det z_4=2^{1-s}$. It follows, provided $x$ is an integer matrix, that $\det D= 2^{nr-n-r+2-s}$.

We then focus on the other terms $\frac12 G_S(\I-S)S^i$. These may be attacked in a similar fashion, keeping $S^i$ outside the deliberations as a factor from the right. This gives us $z_3$. To obtain $x$, we multiply with $z_4^{-1}$ from the left and it is then obvious that $x$ is an integer matrix. 

In the singular case,   $d_0=\frac12(n+r-2)$ since $n+r=(x+y)s$ is even and then Proposition~\ref{jj22} gives that $\det D=2^{nr-n-r+2-s+2f}$.

\begin{prop}No block
$\left(\begin{array}{cc}0&4\\-4&0\end{array}\right)$ appears in the block
diagonal of $H_S$ when the latter is singular.
\end{prop}

\bigskip

\subsubsection{The block diagonal form of a non-regular $H_D$ }
Let us consider the Dipper-Donkin singular case:

\noindent
 Let $n-1=x\cdot s$ and $r+1=y\cdot s$.
Let $$G_{D}=\left(\begin{array}{ccccccc}\I_{s-1}&0&0&0&\dots&0&0\\0&1&0&0&\dots&0&0\\0&0&\I_{s-1}&0&\dots&0&0
\\0&0&0&1&\dots&0&0\\
\vdots&\vdots&\vdots&\vdots&\vdots&\vdots&\vdots\\0&0&0&0&\dots&1&0\\
\I_{s-1}&0&\I_{s-1}&0&\dots&0&\I_{s-1}\end{array}\right).$$

This is a matrix whose diagonal is made up of $y$ blocks of the $(s-1)\times
(s-1)$ identity matrix separated by $(y-1)$ blocks of $1$'s ($1\times 1$
identity matrix) as indicated. The $0$'s represent either $(s-1)\times(s-1)$,
$1\times (s-1)$,  or $(s-1)\times 1$ blocks of zeros. 

We proceed in analogy with the FRT case and introduce $\widehat G_D$
as the analogue of $\widehat G_S$, and we then multiply the two sides of
(\ref{khh}) (case of $H=H_D$) by $\widehat G_D$ from  the left. 
\\
It follows easily that in the present case, $G_{D}(\I-X^{n-1})$ is a matrix
whose bottom $(s-1)$ rows are zero, and hence $G_{D}A_NF$ is a matrix whose
bottom $(s-1)$ rows are zero. We will later prove that if $\widehat L$ denotes
the the top left $(r-(s-1))\times(r-(s-1)$ matrix in $G_{D}A_NF$ then
$\det\widehat L=1$.  {Again, this gives the only unknown piece of
$\det L$.}

We must now examine the effect of multiplying the bottom block row in $K_2$ by $G_{D}$. It is clear that this will result in an integer matrix, and thus we need only concern ourselves with $G_{D}({\mathbb I}-(\I-X^{n-1})(\I-X)^{-1})(-N)^{-1}$. In particular, the bottom $s-1$ rows coincide with those of $G_{D}(-N^{-1})$, and if $z_4$ denotes the  rightmost $(s-1)\times(s-1)$ block of that row, then $\det z_4=1$ follows easily. 

\bigskip

Let us then turn to $\widehat L$, which easily is seen to equal the result of removing the bottom $s-1$ rows and rightmost $s-1$ columns from $X+X^2+\cdots+X^{n-1}$. We do the following elementary column operations: Let $c_i$ denote the $i$th column. Subtract $c_2$ from $c_1$, then $c_3$ from $c_2$, etc until $c_n$ is subtracted from $c_{n-1}$. The resulting matrix has the following form:

$${\mathbb I}_{r-s+1}-\sum_{\alpha=1}^{n-s-1}E_{r-n+2+\alpha,\alpha}-\sum_{\beta=1}^{r-n-s+2}E_{\beta, \beta+n-1}-\sum_{\gamma=2}^{s}E_{r-n-s+1+\gamma,r-s+1}.$$

\medskip

Let us say that this matrix is determined by the data $(x,y)$. In case $2x>y$ we can immediately remove the $1$'s below the diagonal and obtain an upper triangular matrix with $1$'s on the diagonal. The case $2x=y$ is of course not possible. Let us then consider the case $2x<y$. Here we again add top rows to remove the elements corresponding to the term $-\sum_{\alpha=1}^{n-s-1}E_{r-n+2+\alpha,\alpha}$. After that, we can ignore the first rows and columns and are in a case corresponding to the data $(x,y-x)$. If $3x>y$ we are done, and otherwise we reduce again in the place of the $y$. After a finite number of steps we are done.

 \begin{cor} In all cases, regular as well as singular, any non-trivial block of the block diagonal form of any  $H_D$ is
of the
form $\left(\begin{array}{cc}0&1\\-1&0\end{array}\right)$.
\end{cor}

\begin{rem}
This corrects in particular a part of the proof of Theorem~3.1 in \cite{jz-dd}.
\end{rem}

\noindent
Similarly we get

 \begin{cor} Any non-trivial block of the block diagonal form of any  $H_{\co}$
or  $H_{\cd}$ is
of the
form $\left(\begin{array}{cc}0&1\\-1&0\end{array}\right)$.
\end{cor}

\medskip

\subsection{The degree of the extended algebra $\mathcal{P}_q$}\label{sect:full}

Consider a skew symmetric block matrix $H_{\mathcal
P}$ given by 
\begin{equation*}H_{\mathcal
P}=\left(\begin{array}{cccccccc}
0&M_b&M_b&M_b&\cdots&M_b&\I&\hat{E}_1\\N_b&0&M_b&M_b&\cdots&M_b&\I&\hat{E}_2
\\N_b&N_b&0&M_b&\hdots&M_b&\I&\hat{E}_3\\N_b&N_b&N_b&0&
\cdots&M_b&\I&\hat{E}_4
\\\vdots&\vdots&\vdots&\vdots&\cdots&\vdots&\vdots&\vdots
\\N_b&N_b&N_b&N_b&\cdots&0&\I&\hat{E}_n\\
-\I&-\I&-\I&-\I&\hdots&-\I&0&0\\
-\hat{E}_1^t&-\hat{E}_2^t&-\hat{E}_3^t&-\hat{E}_4^t&\hdots&-\hat{E}_n^t&0&0
\end{array}\right),
\end{equation*}
where for the time being $M_b=b\cdot M_r$  is an arbitrary integer multiple of the previously introduced matrix $M_r$ (see p. \pageref{12}). Furthermore, 
 $\I$ denotes the identity matrix of order $r$ and $\hat{E}_\alpha$ are the $r \times n$ matrices defined by
$\hat{E}_\alpha=\sum_{s=1}^{r}E_{s,\alpha}$ for $\alpha=1,\dots,n$. Here we consider the matrix units
$E_{s,i}$ as $r\times n$ (!) matrices with (as usual) a single non-zero entry at
position $(s,i)$. 

We subtract $M_b$ times the (block) column with the $\I$'s from the columns
$2,\dots,n$ and likewise add $N_b$ times the row with the $-\I$'s from the rows
$2,\dots,n$. We further subtract $\hat{E}_1$ times the $I$-column from the last column which results in 
$(0,\hat{E}_2-\hat{E}_1,\hat{E}_3-\hat{E}_2, \dots , \hat{E}_n-\hat{E}_{n-1},0,0 )^t$. We can then make additional  column operations inside this last block column
so that the effect on the blocks $\hat{E}_2,\hat{E}_3,\dots,\hat{E}_n$ are annihilated (specifically we add to the first column   the sum of all the other $r-1$ columns).  
Finally, at this stage, we make the analogous row operations. The net effect is
then a matrix

\begin{equation*}H^{(1)}_{\mathcal P}=Q H_{\mathcal
P}Q^t=\left(\begin{array}{cccccccc}
0&0&0&0&\cdots&0&\I&0\\0&A_b&-N_b&-N_b&\cdots&-N_b&\I&\hat{E}_2\\0&-M_b&A_b&-N_b&\hdots&-N_b&\I&\hat{E}_3\\0&-M_b
&-M_b&A_b&\cdots&-N_b&\I&\hat{E}_4
\\\vdots&\vdots&\vdots&\vdots&\cdots&\vdots&\vdots&\vdots\\0&-M_b&-M_b&-M_b&\cdots&A_b&\I
&\hat{E}_m\\
-\I&-\I&-\I&-\I&\hdots&-\I&0&0\\
0&-\hat{E}_2^t&-\hat{E}_3^t&-\hat{E}_4^t&\hdots&-\hat{E}_m^t&0&0
\end{array}\right),
\end{equation*}
where $A_b=-N_b-M_b$.

We can now subtract the first block row from the other rows and analogously for
the first block column. In this way a total of $r$ blocks of the form
$\left(\begin{array}{cc}0&1\\-1&0\end{array}\right)$ splits off. We then use the
$\hat{E}_i$'s to subtract the  last columns and  last rows in the $A_b,-M_b,-N_b$ terms.
Begin by using $\hat{E}_2,\hat{E}_2^t$, then use $\hat{E}_3,\hat{E}_3^t$, ect. In this way, $n-1$ blocks
of the same form as before split off. Finally,  owing to the removal of
$\hat{E}_1$, in  the cloumn with number $(n+1)r-1$ there are only $0$'s, and analogously, in  the row $(n+1)r+1$, and in this way,  one trivial $2\times2$ matrix splits off. What remains is to consider a matrix of  
the form  

\begin{equation}\label{68}H^{(2)}_{\mathcal
P}=\left(\begin{array}{cccccc}
A&-b\cdot N_{r-1}&-b\cdot N_{r-1}&-b\cdot N_{r-1}&\cdots&-b\cdot N_{r-1}\\-b\cdot M_{r-1}&A&-b\cdot N_{r-1}&-b\cdot N_{r-1}&\cdots&-b\cdot N_{r-1}\\-b\cdot M_{r-1}&-b\cdot M_{r-1}&A&-b\cdot N_{r-1}&\hdots&-b\cdot N_{r-1}\\-b\cdot M_{r-1}&-b\cdot M_{r-1}&-b\cdot M_{r-1}&A&
\cdots  &-b\cdot N_{r-1}
\\\vdots&\vdots&\vdots&\vdots&\cdots&\vdots\\-b\cdot M_{r-1}&-b\cdot M_{r-1}&-b\cdot M_{r-1}&-b\cdot M_{r-1}&\cdots&A
\end{array}\right),
\end{equation}
where $A=-b\cdot M_{r-1}-b\cdot N_{r-1}$.  This matrix is an $(n-1)\times(n-1)$ block matrix in which the 
blocks are of size $(r-1)\times(r-1)$.   
\bigskip

We now assume that $H_{\mathcal P}$ is the defining matrix of the extended 
algebra ${\mathcal P}_q$ introduced in \S \ref{se:setup}. {This is possible if we use a
Dipper-Donkin basis (see Definition~\ref{dddef}). Equivalently, we set
$M=2M_r$. }  Then $H^{(2)}_{\mathcal
P}$ in (\ref{68}) is of the form $$H^{(2)}_{\mathcal
P}=-2H_{\cd} \; ,$$
where $H_{\cd}$ is the matrix introduced at page \pageref{cases}.

Let $r=ys$ and $x=xs$ with $x,y$ relatively prime. Then the corank
$\textrm{corank}(H_{\cd})$ of $H_{\cd}$ based on a $(r-1)\times (n-1)$
configuration has been determined by Proposition~\ref{coranks} as
$\textrm{corank}(H_{\cd})=s-1$.

Then we have obtained 

\begin{prop}The non-trivial blocks of the  matrix $H_{\mathcal{P}}$ of an $n\times r$
quantized extended algebra ${\mathcal P}_q$ are: $(n+r-1)$  blocks {of the form}
$\left(\begin{array}{cc}0&1\\-1&0\end{array}\right)$  together with
$\frac12((n-1)(r-1)-s+1)$ blocks of the form
$\left(\begin{array}{cc}0&2\\-2&0\end{array}\right)$. 
\end{prop}

\begin{cor}\label{correct}
If $n=r$, the non-trivial blocks of the matrix $H_{\mathcal{P}}$ of the
quantized extended algebra ${\mathcal P}_q$ are: $(2n-1)$  blocks {of the form}
$\left(\begin{array}{cc}0&1\\-1&0\end{array}\right)$  together with
$\frac12 (n-1)(n-2)$ blocks of the form
$\left(\begin{array}{cc}0&2\\-2&0\end{array}\right)$. 
\end{cor}

\begin{rem}Notice that when $n=r$, the total number of blocks is $\frac12
n(n+1)$. Corollary~\ref{correct} corrects the distribution of
the two kinds of blocks as given in \cite[Theorem~11.2]{jz1}
\end{rem}

\bigskip


\section{The   quasi-commutation matrix $\Lambda$ and its inverse }\label{se:Linv}

In this section we analyze the matrix $\Lambda$, introduced in Proposition~\ref{propo}, which encodes the commutation relations among the quantum minors $\chi_{\alpha j}$.
As in \eqref{THT}:
$$
\Lambda= {\mathbb T}^t H{\mathbb T} ,
$$
where 
the  matrix ${\mathbb T}$ was defined  in \eqref{matrixTT} and  can be written as a matrix made of $n \times n$ blocks
of order $r$ in the following form:
\be \label{TT}
{\mathbb T}=
\begin{pmatrix}
\I & T & T^2 & \cdots & T^{n-1} \\
0&\I & T & \cdots &T^{n-2}  \\
0 &0 & \I & \cdots &\vdots  \\
\vdots &\vdots& \cdots & \ddots&T \\
0 &0 & \cdots & 0 &\I
\end{pmatrix} \; . 
\ee
Here, $T $ is the $r \times r$  matrix introduced previously in \eqref{matrixT}.
It is evident that $det({\mathbb T})=1$ and the inverse matrix is given by

\be\label{Tinv}
{\mathbb T}^{-1}=
\begin{pmatrix}
\I & -T & 0 & \cdots & 0 \\
0&\I & -T & \cdots &0  \\
0 &0 & \I & \cdots &\vdots  \\
\vdots &\vdots& \cdots & \ddots&-T \\
0 &0 & \cdots & 0 &\I
\end{pmatrix}
\ee
\\

As an example we write here the case of $n=r=3$:
$$
{\mathbb T}={\small \left(
\begin{array}{ccc|ccc|ccc}
 1 & 0 & 0 & 0 & 1 & 0 & 0 & 0 & 1 \\
 0 & 1 & 0 & 0 & 0 & 1 & 0 & 0 & 0 \\
 0 & 0 & 1 & 0 & 0 & 0 & 0 & 0 & 0 \\ \hline
 0 & 0 & 0 & 1 & 0 & 0 & 0 & 1 & 0 \\
 0 & 0 & 0 & 0 & 1 & 0 & 0 & 0 & 1 \\
 0 & 0 & 0 & 0 & 0 & 1 & 0 & 0 & 0 \\ \hline
 0 & 0 & 0 & 0 & 0 & 0 & 1 & 0 & 0 \\
 0 & 0 & 0 & 0 & 0 & 0 & 0 & 1 & 0 \\
 0 & 0 & 0 & 0 & 0 & 0 & 0 & 0 & 1
\end{array}
\right) \; ,  \quad {\mathbb T}^{-1}=
\left(
\begin{array}{ccc|ccc|ccc}
 1 & 0 & 0 & 0 & -1 & 0 & 0 & 0 & 0 \\
 0 & 1 & 0 & 0 & 0 & -1 & 0 & 0 & 0 \\
 0 & 0 & 1 & 0 & 0 & 0 & 0 & 0 & 0 \\ \hline
 0 & 0 & 0 & 1 & 0 & 0 & 0 & -1 & 0 \\
 0 & 0 & 0 & 0 & 1 & 0 & 0 & 0 & -1 \\
 0 & 0 & 0 & 0 & 0 & 1 & 0 & 0 & 0 \\ \hline
 0 & 0 & 0 & 0 & 0 & 0 & 1 & 0 & 0 \\
 0 & 0 & 0 & 0 & 0 & 0 & 0 & 1 & 0 \\
 0 & 0 & 0 & 0 & 0 & 0 & 0 & 0 & 1
\end{array}
\right) }.
$$
\medskip

In the full rank case, the invertibility of the matrix $H$ \eqref{matrixH} implies that of $\Lambda$:
$$
\Lambda^{-1}= ({\mathbb T}^{-1}) H^{-1} ({\mathbb T}^{-1})^t,
$$
or in block ($r \times r$ matrices) components  
\be \label{lambda_inv}
\Lambda^{-1}_{\alpha \beta }= H^{-1}_{\alpha \beta } -  H^{-1}_{\alpha, \beta+1} T^t - T H^{-1}_{\alpha+1,\beta} + T H^{-1}_{\alpha+1,\beta+1} T^t ,\quad \alpha,\beta=1, \dots , n,
\ee
where the second  term $H^{-1}_{\alpha, \beta+1} T$ appears only when $\beta<n$,  and with  analogous properties  for the other terms.  With the aim of determining  compatible pairs $(\Lambda_{\mathcal{M}},\tilde{B}_{\mathcal{M}})$ as in \eqref{cc}, 
we study below the explicit form of the matrix $\Lambda^{-1}$ in the two particular cases $H=H_D$ and $H=H_S$: 

\subsection{The inverse matrix $\Lambda^{-1}$}\label{se:lambda_inv}

\subsubsection{Dipper-Donkin general full rank case}

Let us consider the case of $H=H_D$, for $ g.c.d.(n-1,r+1)=1$. In what follows,  we avoid writing the subscripts $_D$ and $_r$ to matrices to lighten the notation.

\medskip
\noindent 
We collect first some useful formulas immediately derivable from the very definitions of the matrices $X, T:$ 
\begin{enumerate}[(i)]
 \item $F^{-1}N=\frac{\I-X}{X-X^n}$ ;
\item for $r=2m$,  $(\I+X)^{-1}=- \sum_{i=1}^mX^{2i-1}$ (cf. \eqref{def_a}). Indeed $$-(\I + X) \sum_{i=1}^mX^{2i-1} =
 - \sum_{i=1}^{r} X^{i}= - \sum_{i=0}^{r} X^{i}+\I =\I \; ; $$
\item $(\I-T)T^t=X^{-1}(\I-T)+E_{rr}=T^t-\I+E_{rr}$,  and $X^{-1}(\I-T)=-\I+T^t$ ;
\item \label{iv} $TX^{-1}=\I-E_{rr}$ and  $X^{-1}N^{-1} -N^{-1} T^t =  E_{rr}$.
\end{enumerate}~
\medskip

Then we compute the different  blocks of the inverse matrix $\Lambda^{-1}$:
\begin{itemize}

\item If $n> \alpha>\beta +1$ then all terms in $\Lambda^{-1}$ appear. Using in this range  $$
H^{-1}_{\alpha \beta}= X^{n - \alpha + \beta -1} (\I-X)(X - X^n)^{-1} N^{-1}$$ and $H^{-1}_{\alpha +1,\beta}=X^{-1}H^{-1}_{\alpha, \beta}$ one easily gets 
\be\label{la_ab}
\Lambda^{-1}_{\alpha, \beta}=E_{rr}\left(\frac{X^{n+\beta-\alpha}(\I-X)}{X-X^n}\right)E_{rr}=\left(\frac{X^{n+\beta-\alpha}(\I-X)}{X-X^n}\right)_{rr}E_{rr}.
\ee
Indeed  from \eqref{lambda_inv},
\begin{eqnarray*}
\Lambda^{-1}_{\alpha \beta }&=&  (\I - T X^{-1}) H^{-1}_{\alpha, \beta} - (\I - T X^{-1}) H^{-1}_{\alpha, \beta+1} T^t
= E_{rr} (H^{-1}_{\alpha, \beta} -  H^{-1}_{\alpha, \beta+1} T^t) \\
&=& E_{rr} X^{n - \alpha + \beta } (\I -X)(X - X^n)^{-1} [X^{-1} N^{-1}- N^{-1}T^t] \\ &=& 
E_{rr} X^{n - \alpha + \beta } (\I-X)(X - X^n)^{-1} E_{rr} \; ,
\end{eqnarray*}
from which \eqref{la_ab} follows.
\item We define $P_n:=-\left(\frac{X^{n-1}(\I-X)}{X-X^n}\right)$. Then 
$$H^{-1}_{\alpha,\alpha}=(P_n+\I)(\I-T) \quad ; \quad H^{-1}_{\alpha,\alpha-1}=-P_n X^{-1}N^{-1},$$ so in particular we observe that $H^{-1}_{\alpha,\alpha}$ and $H^{-1}_{\alpha,\alpha-1}$ do not depend on the block index $\alpha$. Furthermore, notice
that $(P_n)_{rr}=-1$. Indeed, from \eqref{def_a} we have that
$P_n = -X^{-1} [X^{n-1}+X^{2(n-1)}+ \dots + X^{a(n-1)}]$. Next, for each $a'<a$, $X^{a'(n-1)}\neq X$, since by our assumptions, $a$ is the smallest positive integer such that $a(n-1) \equiv 1 ~\mbox{mod}(r+1).$ Similar reasoning shows that for all  $a'<a$, $X^{a'(n-1)}\neq X^r$ because otherwise we would have $(a-a')(n-1) \equiv 1 ~\mbox{mod}(r+1).$
Hence, from \eqref{Xl}, the only term in $P_n$ which has a non-zero $r,r$ component is $-X^{-1} X^{a(n-1)}= -\I$. 
\\
\noindent
Then, for $\alpha \leq n-1$, we compute
\be\label{la_aa-1}
\Lambda^{-1}_{\alpha,\alpha-1}=\I-T^t. \ee
Indeed, by using $E_{rr}(\I+P_n)E_{rr}=0$ and 
$$
H_{\alpha+1,\alpha}^{-1}=X^{-1}H_{\alpha,\alpha}^{-1}-X^{-1}(\I-T)=X^{-1}H_{\alpha,\alpha}^{-1}+\I-T^t
$$
we have
\begin{eqnarray*}
&\Lambda^{-1}_{\alpha, \alpha -1 }&=  (\I - T X^{-1}) H^{-1}_{\alpha, \alpha-1} - (\I - T X^{-1}) H^{-1}_{\alpha, \alpha} T^t +T X^{-1}N^{-1}T^t \\
&&= E_{rr} H^{-1}_{\alpha, \alpha-1} - E_{rr}  H^{-1}_{\alpha, \alpha} T^t+(\I-E_{rr})(E_{rr}+X^{-1}N^{-1}) \\
&&= -E_{rr}P_n X^{-1}N^{-1}-E_{rr}(P_n+1)(T^t-\I+E_{rr})+ X^{-1}N^{-1}-E_{rr}X^{-1}N^{-1}\\
&&= -E_{rr}P_n M^{-1} + E_{rr}P_n M^{-1} + E_{rr} M^{-1} + M^{-1} -E_{rr} M^{-1} = \I-T^t.
\end{eqnarray*}

\item We observe that $H_{\alpha,\alpha}^{-1}$ is skew symmetric.
We use
$$H_{\alpha,\alpha+1}^{-1}=H_{\alpha,\alpha}^{-1}(X^{-1})^t+(\I-T^t)(X^{-1})^t=H_{\alpha,\alpha}^{-1}(X^{-1})^t-\I+T ~$$
and compute $\Lambda_{\alpha,\alpha}^{-1}$ for $\alpha \leq n-1$:
\begin{eqnarray*}
\Lambda^{-1}_{\alpha \alpha}&=& H^{-1}_{\alpha, \alpha} - H^{-1}_{\alpha, \alpha} (X^{-1})^t T^t +T^t -TT^t-
T X^{-1}   H^{-1}_{\alpha, \alpha} -T+ T T^t+T  H^{-1}_{\alpha, \alpha}T^t \\
&=& H^{-1}_{\alpha, \alpha} E_{rr}+T^t  +E_{rr}  H^{-1}_{\alpha, \alpha} - H^{-1}_{\alpha, \alpha} -T+ T  H^{-1}_{\alpha, \alpha}T^t.
\end{eqnarray*}
Next,
\begin{eqnarray*}
H^{-1}_{\alpha, \alpha} E_{rr}+(E_{rr}-\I)  H^{-1}_{\alpha, \alpha}+ T  H^{-1}_{\alpha, \alpha}T^t
&=&
H^{-1}_{\alpha, \alpha} E_{rr} -TX^{-1} (P_n+\I)(\I-T)  \\&& + T(P_n +\I) (X^{-1} (\I- T) +E_{rr})
\\ &=& \left(H^{-1}_{\alpha, \alpha}+  T (P_n+\I) \right)E_{rr} ,
\end{eqnarray*}
(where we used that $P_n$ and $X^{-1}$ commute), and therefore we conclude that
for $\alpha \leq n-1$: 
$$ \Lambda_{\alpha,\alpha}^{-1}=T^t-T +(P_n+\I-P_nT+TP_n)E_{rr}~ . $$
\\
Define the matrix $Q$ by $T=X+Q$, i.e. $Q=\sum_{i=1}^r E_{ri}$. Then, clearly,  $(P_n+\I-P_nT+TP_n)E_{rr}=(P+\I-PQ+QP)E_{rr}$. We observe that $E_{rr}X = -Q$ and $QE_{rr}=E_{rr}$, so that 
 $(P_n+\I-P_nQ+QP_n)E_{rr}= E_{rr} + QP_n E_{rr}$. Furthermore  $QP_nE_{rr}= -E_{rr}XP_n E_{rr}$, with $(XP_n)_{rr}=1$.  We conclude that $(P_n+\I-P_nT+TP_n)E_{rr}=0$ and  
$$
\Lambda_{\alpha,\alpha}^{-1}=T^t-T \, , \quad \forall \alpha \leq n-1 \, .$$
\item Evidently, $\Lambda_{n,n}^{-1}=H_{n,n}^{-1}=(P_n+\I)(\I-T)$.
\item For $\beta<n-1$,
$\Lambda_{n\beta}^{-1}=H_{n\beta}^{-1}-H_{n,\beta+1}^{-1}T^t=-X^{\beta}\frac{(\I-X)}{X-X^n}E_{rr}
$.
\item
$\Lambda_{n,n-1}^{-1}=-X^{n-1}\frac{(\I-X)}{X-X^n}E_{rr}
-T^t+\I-E_{rr}=P_nE_{rr}-T^t+\I-E_{rr}$.
\end{itemize} ~

Summarizing, we have the following:
\begin{eqnarray*}
\Lambda^{-1}_{\alpha \beta}= \left\{
\begin{array}{lll}
\left(\frac{X^{n+\beta-\alpha}(\I-X)}{X-X^n}\right)_{rr}E_{rr} & \mbox{ if }& n> \alpha>\beta +1
\\ \vspace{3pt}
\I-T^t & \mbox{ if }& \alpha \neq n,~\beta=\alpha-1
 \\ \vspace{3pt}
T^t-T& \mbox{ if }& \alpha = \beta \neq  n
\\ \vspace{3pt}
(P_n+\I)(\I-T) & \mbox{ if }& \alpha = \beta = n
\\ \vspace{3pt}
-X^{\beta}\frac{(1-X)}{X-X^n}E_{rr} & \mbox{ if }& \alpha =  n,~\beta<n-1
\\ \vspace{3pt}
P_nE_{rr}-T^t+\I-E_{rr} & \mbox{ if }& \alpha =  n,~\beta=n-1
\\ \vspace{3pt}
-(\Lambda^{-1}_{\beta \alpha})^t & \mbox{ if }& \alpha <\beta
\end{array} \right..
\end{eqnarray*}

\medskip

\subsubsection*{Dipper-Donkin, case $n=r$ even}

In the particular case in which $n=r=2m$, the expression of $\Lambda^{-1}$ just determined can be simplified further, thanks to the fact that here $X^n=X^{-1}$ and $(\I+X)^{-1}= -X-X^3 \dots -X^{r-1}$.
\begin{itemize}
\item 
For $ n> \alpha>\beta +1$, we have $\Lambda^{-1}_{\alpha \beta}= (X^{\beta-\alpha}(X +X^3+\dots X^{r-1}))_{rr}E_{rr} \, .$
Recalling the expression of $X^l$  from equation \eqref{Xl}, we see that $(X^l)_{rr}$ is non zero only for the values $l=0,1,r+1$. Thus,
\begin{eqnarray*}
\Lambda^{-1}_{\alpha \beta}= \left\{
\begin{array}{lll}
- E_{rr} &\mbox{ if }& \alpha - \beta \mbox{ even }
\\
+E_{rr} &\mbox{ if }& \alpha - \beta \mbox{ odd }
\end{array}
\right. .
\end{eqnarray*}
Indeed,  if $\alpha-\beta$ is even,  the sum $X^{-(\alpha -\beta)}(X +X^3+\dots X^{r-1})$  contains the term $X$, while if  odd, it contains $\I$.

\item The term $P_n$ here is simply the sum of the even powers of $X:$
\be\label{Pn=r}
P_n= \I+X^2+\dots X^{r-2}. \ee

\item For $\alpha=n, ~\beta<n-1$, we need to analyse the term 
$$- X^\beta \frac{\I-X}{X-X^r}=X^\beta(X^2+X^4+\dots X^{r-2}+X^r)=: L\, .$$
If $\beta$ is even, then 
$$
\begin{matrix}
L= \underbrace{X^{\beta+2}+X^{\beta+4}+\dots X^r}_{even} +\underbrace{ X+X^3+\dots X^{\beta-1}}_{odd}
\end{matrix}
$$
while if $\beta$ is odd, 
$$
\begin{matrix}
L= \underbrace{X^{\beta+2}+X^{\beta+4}+\dots \I}_{odd} +\underbrace{ X^2+X^4+\dots X^{\beta-1}}_{even}.
\end{matrix}
$$
Now let us consider the last column of $L$. Accordingly to \eqref{Xl}, we have $(X^l)_{ir}=\delta_{i,r-l}-\delta_{i,r-l+1}$ (the first term disappearing for $i= r$), so that 
\be\label{pot}
L_{ir}=
\left\{ 
\begin{array}{ll}
(-1,1,-1,1,\dots  ,0, \dots 1,-1,1,-1,1,-1)^t  & \mbox{ if } \beta \mbox{ even }  \vspace{3pt}
\\ 
(1,-1,1,-1,\dots ,0, \dots -1,1,-1, 1,-1,1)^t  & \mbox{ if } \beta \mbox{ odd} 
\end{array}
\right.,
\ee
i.e. $L_{ir}$ is a column of alternating $1,-1$ starting and finishing with $-1$ if $\beta$ is even, with $1$ if $\beta$ is odd,  and with a 
single $0$ at the position  of the row $r - \beta$.
Hence we conclude that the only non zero column of $\Lambda^{-1}_{n,\beta}, ~\beta<n-1$ is
$$
(\Lambda^{-1}_{n,\beta})_{ir}= \left\{ 
\begin{array}{ll}
(-1,1,-1,1,\dots  ,0, \dots 1,-1,1,-1,1,-1)^t  & \mbox{ if } \beta \mbox{ even }  \vspace{3pt}
\\ 
(1,-1,1,-1,\dots ,0, \dots -1,1,-1, 1,-1,1)^t  & \mbox{ if } \beta \mbox{ odd} 
 \end{array}
\right.
\, .
$$
\item From \eqref{Pn=r} above, we have that ${\Lambda^{-1}_{n,n-1}}=(\I-T^t)+ (P_nE_{rr}-E_{rr})$, where the  matrix $P_nE_{rr}-E_{rr}$ has the last (and only non-zero) column given by
$$
(0,1,-1,1,-1 \dots 1,-1,0)^t
$$ 
by a reasoning analogous to the one used for \eqref{pot}. 
\end{itemize}
\bigskip

As an example, we write the matrices $H, \Lambda$ and their inverses for the case $n=r=4$.
 The matrix $H_D$  is given  by \eqref{matrixH}, with 
$$
A=0\, , \quad 
N=\left(\begin{array}{cccc}-1&-1&-1&-1\\0&-1&-1&-1\\0&0&-1&-1 \\0&0&0&-1
\end{array}\right) \textrm{ and }
M=\left(\begin{array}{cccc}1&0&0&0\\1&1&0&0\\1&1&1&0 \\1&1&1&1
\end{array}\right).
$$
 Its inverse is

\be\label{Hinv4}
{\small  H^{-1}=  \left(
\begin{array}{cccc|cccc|cccc|cccc}
 0 & 0 & -1 & 1 & -1 & 1 & 0 & -1 & 1 & -1 & 1 & 0 & -1 & 1 & -1 & 1 \\
 0 & 0 & 0 & -1 & 1 & -1 & 1 & 0 & -1 & 1 & -1 & 1 & 0 & -1 & 1 & -1 \\
 1 & 0 & 0 & 0 & -1 & 1 & -1 & 1 & 0 & -1 & 1 & -1 & 1 & 0 & -1 & 1 \\
 -1 & 1 & 0 & 0 & 0 & -1 & 1 & -1 & 1 & 0 & -1 & 1 & -1 & 1 & 0 & -1 \\
 \hline
 1 & -1 & 1 & 0 & 0 & 0 & -1 & 1 & -1 & 1 & 0 & -1 & 1 & -1 & 1 & 0 \\
 -1 & 1 & -1 & 1 & 0 & 0 & 0 & -1 & 1 & -1 & 1 & 0 & -1 & 1 & -1 & 1 \\
 0 & -1 & 1 & -1 & 1 & 0 & 0 & 0 & -1 & 1 & -1 & 1 & 0 & -1 & 1 & -1 \\
 1 & 0 & -1 & 1 & -1 & 1 & 0 & 0 & 0 & -1 & 1 & -1 & 1 & 0 & -1 & 1
 \\ \hline
 -1 & 1 & 0 & -1 & 1 & -1 & 1 & 0 & 0 & 0 & -1 & 1 & -1 & 1 & 0 & -1 \\
 1 & -1 & 1 & 0 & -1 & 1 & -1 & 1 & 0 & 0 & 0 & -1 & 1 & -1 & 1 & 0 \\
 -1 & 1 & -1 & 1 & 0 & -1 & 1 & -1 & 1 & 0 & 0 & 0 & -1 & 1 & -1 & 1 \\
 0 & -1 & 1 & -1 & 1 & 0 & -1 & 1 & -1 & 1 & 0 & 0 & 0 & -1 & 1 & -1 \\
 \hline
 1 & 0 & -1 & 1 & -1 & 1 & 0 & -1 & 1 & -1 & 1 & 0 & 0 & 0 & -1 & 1 \\
 -1 & 1 & 0 & -1 & 1 & -1 & 1 & 0 & -1 & 1 & -1 & 1 & 0 & 0 & 0 & -1 \\
 1 & -1 & 1 & 0 & -1 & 1 & -1 & 1 & 0 & -1 & 1 & -1 & 1 & 0 & 0 & 0 \\
 -1 & 1 & -1 & 1 & 0 & -1 & 1 & -1 & 1 & 0 & -1 & 1 & -1 & 1 & 0 & 0
\end{array} 
\right) }. \ee

The matrix $ \Lambda={\mathbb T}^t H{\mathbb T}$ is given by

$$  {\small \Lambda  =  \left(
\begin{array}{cccc|cccc|cccc|cccc}
 0 & 0 & 0 & 0 & 1 & 0 & 0 & 0 & 1 & 1 & 0 & 0 & 1 & 1 & 1 & 0 \\
 0 & 0 & 0 & 0 & 1 & 1 & 0 & 0 & 1 & 2 & 1 & 0 & 1 & 2 & 2 & 1 \\
 0 & 0 & 0 & 0 & 1 & 1 & 1 & 0 & 1 & 2 & 2 & 1 & 1 & 2 & 3 & 2 \\
 0 & 0 & 0 & 0 & 1 & 1 & 1 & 1 & 1 & 2 & 2 & 2 & 1 & 2 & 3 & 3 \\ 
\hline
- 1 &- 1 &- 1 &- 1 & 0 & -1 & -1 & -1 & 1 & 0 & -1 & -1 & 1 & 1 & 0 & -1 \\
 0 & -1 & -1 & -1 & 1 & 0 & -1 & -1 & 2 & 2 & 0 & -1 & 2 & 3 & 2 & 0 \\
 0 & 0 & -1 & -1 & 1 & 1 & 0 &- 1 & 2 & 3 & 2 & 0 & 2 & 4 & 4 & 2 \\
 0 & 0 & 0 & -1 & 1 & 1 & 1 & 0 & 2 & 3 & 3 & 2 & 2 & 4 & 5 &4 \\ 
\hline
- 1 & -1 & -1 &- 1 & -1 & -2 & -2 & -2 & 0 & -1 & -2 & -2 & 1 & 0 &- 1 &- 2 \\
 -1 & -2 & -2 & -2 & 0 &- 2 & -3 & -3 & 1 & 0 & -2 & -3 & 2 & 2 & 0 & -2 \\
 0 & -1 & -2 & -2 & 1 & 0 & -2 & -3 & 2 & 2 & 0 & -2 & 3 & 4 & 3 & 0 \\
 0 & 0 &- 1 & -2 & 1 & 1 & 0 & -2 & 2 & 3 & 2 & 0 & 3 & 5 & 5 & 3 \\ 
\hline
- 1 & -1 & -1 & -1 & -1 & -2 & -2 & -2 & -1 & -2 & -3 & -3 & 0 & -1 & -2 & -3 \\
 -1 & -2 & -2 & -2 & -1 & -3 & -4 & -4 & 0 & -2 & -4 & -5 & 1 & 0 & -2 & -4 \\
 -1 & -2 & -3 & -3 & 0 & -2 & -4 & -5 & 1 & 0 & -3 & -5 & 2 & 2 & 0 & -3 \\
 0 & -1 & -2 & -3 & 1 & 0 &- 2 & -4 & 2 & 2 & 0 & -3 & 3 & 4 & 3 & 0
\end{array}
\right) }$$

It is of full rank, with inverse  given by

$$ \Lambda^{-1}= {\small \left(
\begin{array}{cccc|cccc|cccc|cccc}
 0 & -1 & 0 & 0 &- 1 & 1 & 0 & 0 & 0 & 0 & 0 & 0 & 0 & 0 & 0 & 0 \\
 1 & 0 &- 1 & 0 & 0 & -1 & 1 & 0 & 0 & 0 & 0 & 0 & 0 & 0 & 0 & 0 \\
 0 & 1 & 0 &- 1 & 0 & 0 &- 1 &1 & 0 & 0 & 0 & 0 & 0 & 0 & 0 & 0 \\
 0 & 0 & 1 & 0 & 0 & 0 & 0 &- 1 & 0 & 0 & 0 & 1 & -1 & 1 & 0 & -1 \\ 
\hline
 1 & 0 & 0 & 0 & 0 & -1 & 0 & 0 & -1 & 1 & 0 & 0 & 0 & 0 & 0 & 0 \\
 -1 & 1 & 0 & 0 & 1 & 0 & -1 & 0 & 0 & -1 & 1 & 0 & 0 & 0 & 0 & 0 \\
 0 & -1 & 1 & 0 & 0 & 1 & 0 & -1 & 0 & 0 & -1 & 1 & 0 & 0 & 0 & 0 \\
 0 & 0 & -1 & 1 & 0 & 0 & 1 & 0 & 0 & 0 & 0 &- 1 & 1 & 0 &-1 &1 \\ 
\hline
 0 & 0 & 0 & 0 & 1 & 0 & 0 & 0 & 0 & -1 & 0 & 0 & -1 &1 & 0 & 0 \\
 0 & 0 & 0 & 0 & -1 & 1 & 0 & 0 & 1 & 0 & -1 & 0 & 0 & -1 & 1 & 0 \\
 0 & 0 & 0 & 0 & 0 & -1 & 1 & 0 & 0 &1 & 0 & -1 & 0 & 0 & -1 & 1 \\
 0 & 0 & 0 & -1 & 0 & 0 & -1 & 1 & 0 & 0 & 1 & 0 & 0 & -1 & 1 & -1 \\ 
\hline
 0 & 0 & 0 & 1 & 0 & 0 & 0 & -1 & 1 & 0 & 0 & 0 & 0 & 0 & -1 & 1 \\
 0 & 0 & 0 & -1 & 0 & 0 & 0 & 0 & -1 & 1 & 0 & 1 & 0 & 0 & 0 & -1 \\
 0 & 0 & 0 & 0 & 0 & 0 & 0 & 1 & 0 & -1 & 1 & -1 & 1 & 0 & 0 & 0 \\
 0 & 0 & 0 & 1 & 0 & 0 & 0 & -1 & 0 & 0 & -1 & 1 & -1 & 1 & 0 & 0
\end{array}
\right).}$$ ~
  
\bigskip

\subsubsection{The FRT case}
We compute the matrix $\Lambda^{-1}$ in \eqref{lambda_inv} when $H=H_S$ is of full rank. We omit the subscripts $_S$ and $_r$.

Recall that $S^r=-\I$, $S^{-1}=S^t$, $F=\frac{\I+S^n}{\I-S}$, $N=-1$ and
$(A-N)^{-1}=\frac12(\I-S)$.  Moreover recall the form of $H^{-1}$ from \eqref{HS_inv} and  that 
$H_{\alpha,\alpha}^{-1}$ is skew symmetric for each $\alpha=1, \dots, n$.
We will use that $T=S+E_{r1}$ and hence $T^t=S^{-1}+E_{1r}$.

\begin{itemize}
\item If $n> \alpha>\beta+1$ then $H^{-1}_{\alpha, \beta+1}=S H^{-1}_{\alpha, \beta}= H^{-1}_{\alpha, \beta} ~S$ and  $H^{-1}_{\alpha+1,\beta}=S^{-1}H^{-1}_{\alpha \beta }$,  so that
\begin{eqnarray*}
\Lambda^{-1}_{\alpha \beta }&=& H^{-1}_{\alpha \beta }( \I-S T^t) - T H^{-1}_{\alpha+1,\beta}( \I -S T^t )
= (\I-TS^{-1})  H^{-1}_{\alpha \beta } (-SE_{1r}) = E_{r1} H^{-1}_{\alpha \beta } E_{1r}\, , 
\end{eqnarray*}
 that is 
\be \Lambda_{\alpha \beta}^{-1}=\left(H_{\alpha \beta}^{-1}\right)_{11}E_{rr}. \ee

\item For
$ \alpha \leq n-1$ we find
\be 
\Lambda_{\alpha \alpha}^{-1}=\frac12(T^t-T).
\ee
Indeed, using the expression of $T$ and $T^t$ in terms of $S$, with some algebra we get
\begin{eqnarray*}
2 \Lambda_{\alpha \alpha}^{-1} &=& \frac{(\I-S)}{(\I+S^{n})} (S^{-1}+S^n +\I +S^{n-1}) + \frac{(\I-S)}{(\I+S^{n})} (\I+S^{n}) E_{1r} + \\ & +& E_{r1} \frac{(\I-S)}{(\I+S^{n})} (S^{n-1} + S^{-1} + (\I+S^{n-1})E_{1r}) \\
&=& (\I-S)(\I+S^{-1})+(\I-S)E_{1,r}+
 E_{r1}(S^{-1}-\I) + \\ && +E_{r1} \frac{(\I-S)(\I+S^{n-1})}{(\I+S^{n})} E_{1r} \, .
\end{eqnarray*}
The last summand vanishes because it coincides with $ E_{r1} H^{-1}_{\alpha \alpha}E_{1r} = (  H^{-1}_{\alpha \alpha})_{11}E_{rr}$ which is zero because of the antisymmetry of $ H^{-1}_{\alpha \alpha}$. The remaining part is promptly
verified to coincide with $T^t-T$.

\item As in the Dipper-Donkin case, $\Lambda_{rr}^{-1}=H_{rr}^{-1}$.
\item For $\alpha \leq n-1$:
\be\label{las_aa-1}\Lambda_{\alpha,\alpha-1}^{-1}=\left(H_{\alpha,\alpha-1}^{-1}\right)_{11}E_{rr}-\frac12(T^t-\I+E_{rr}).\ee 
This formula is derived from the explicit expression of $H^{-1}$ in terms of $S$, similarly to the  computations done just above.
In the particular case of $n=r+1$,  equation \eqref{las_aa-1} reduces to
\begin{eqnarray*}
\Lambda_{\alpha,\alpha-1}^{-1}=\frac12(\I-T^t).
\end{eqnarray*}
\item For $\beta<n-1$:
\begin{eqnarray*}
\Lambda_{n\beta}^{-1}&=& H_{n\beta}^{-1} - H_{n, \beta+1}^{-1} T^t =H_{n\beta}^{-1} (\I- ST^t) = H_{n\beta}^{-1}(\I-S(S^{-1}+E_{1r}))= \\ &=& -H_{n,\beta}^{-1} S E_{1r}
\end{eqnarray*}
so that 
\be
\Lambda_{n\beta}^{-1}= - \frac12 \frac{
S^\beta(\I-S)^2}{(\I+S^n)}E_{1r}.
\ee
\item We compute 
\begin{eqnarray*}
\Lambda_{n,n-1}^{-1}&=& H_{n,n-1}^{-1} - H_{n,n}^{-1} (S^{-1}+E_{1r})  \\ &=&
\frac{1}{2} \frac{(\I-S)}{(\I+S^n)} [S^{n-2} (1-S) -(\I+S^{n-1})(S^{-1}+E_{1r})]
\\ &=& \frac{1}{2} (\I-S^{-1}) - \frac{1}{2} \frac{(\I+S^{n-1})}{(\I+S^n)} (\I-S)E_{1r}
\end{eqnarray*} 
hence concluding, by using  $TE_{1r}=0$ and so $SE_{1r}=-E_{rr}$, that
\be \Lambda_{n,n-1}^{-1}
=-\frac {
(\I+S^{n-1})}{2(\I+S^n)}(E_{1r}+E_{rr})- \frac12(S^{-1}-\I).\ee
\end{itemize}
\medskip

\subsection{The case of non-invertibility} ~

Let $H$ be as in (\ref{matrixH}).
Let $\underline{a}=(\underline{a}_n,\dots,\underline{a}_2,\underline{a}_1)^t$ be a ({column}) vector in ${\mathbb C}^{nr}$, i.e. such that $\forall i=1,\dots,n:\underline{a}_i \in{\mathbb C}^r$.  Suppose that
$H\underline{a}=0$. It follows easily from \eqref{H4} that $\underline{a}$, up to a constant multiple, is determined by  
\begin{equation*}
\forall c=2,\dots,n: \underline{a}_c=X^{c-1}\underline{a}_1,\ ; \ A_NF \underline{a}_1=\underline{0}
\end{equation*} 
By the    assumptions in \S \ref{se:Hmatrix}, $\underline{a}$ is thus completely determined by solutions to the equation $F\underline{a}_1=0$. 
\\
If $\underline{b}=(\underline{b}_n,\dots,\underline{b}_2,\underline{b}_1)$ is a vector defined in analogy to $\underline{a}$, and if  $\Lambda\underline{b}=0$, then, up to a constant multiple, 
\begin{equation*}
\underline{b}={\mathbb T}^{-1}\underline{a},
\end{equation*}
and hence, by \eqref{Tinv}, 
\begin{equation}\label{b_s}
\forall c=2,\dots, n, \quad \underline{b}_c=(X-T)X^{c-2}\underline{a}_1, \quad \textrm{ and } \underline{b}_1=\underline{a}_1.
\end{equation}
\\

\noindent
For the FRT case as well as the Dipper-Donkin case  it holds (also for $n,r$ arbitrary) that 
$$(X-T)\in\textrm{Span}{\{E_{r1},E_{r2,}\dots,E_{rr}\}}.$$

It follows that, in the above notation,  
\begin{prop} \label{prop:bc} In the FRT, as well as the Dipper-Donkin case, there are integers $z_c$, $c=2,\dots,n$ such that
\begin{equation*}
\forall c=2,\dots,n: \underline{b}_c=z_c{\bold e}_{r},
\end{equation*}
where ${\bold e}_{r}$ is the ${r}$th basis vector in the standard basis of ${\mathbb C}^{r}$.
\end{prop}

From the above considerations we  conclude the following results about the centers of the quasi-polynomial algebra generated by the quantum minors $\chi_{\alpha j}\in \mathcal{M}_q $ (see Definition \ref{family_minors}) for $\mathcal{M}_q $ being either the Dipper-Donkin algebra $\mathcal{D}_{q^2}(M(n,r))$ or the FRT algebra $\mathcal{O}_q(M(n,r))$. (For clarity we introduce a comma $\chi_{\alpha j}= \chi_{\alpha, j} $).

\medskip

\subsubsection{Dipper-Donkin; special case}
\begin{prop}\label{prop:z} For $r=n=odd$, the center of the {quasi-}Laurent polynomial algebra generated by the $n^2$ elements $\chi_{\alpha,j}$ $\in \mathcal{D}_{q^2}(M(r,r))$ is generated by
\begin{equation}\label{z-form}
Z:=\prod_{\gamma=1}^n\chi_{\gamma,n}^{(-1)^\gamma}
\prod_{k=1}^{n-1}\chi_{n,k}^{(-1)^k},
\end{equation}
{and its inverse.} 
\end{prop}

\begin{proof}
The equation $F\underline{a}_1=0$ becomes  $X\underline{a}_1=-\underline{a}_1$. If follows easily that up to a constant multiple,  $\underline{a}_1=(1,-1,1,\dots,-1,1)$.
Furthermore, 

\begin{eqnarray}(X-T)&=&-E_{r1}-\dots-E_{rr}\textrm{ and }\label{one}\\(X-T)X^{c-2}= (\I - TX^{-1}) X^{c-1}&=& E_{rr} X^{c-1}= E_{r, c-2}\textrm{ for }c>2. \label{two}
\end{eqnarray}
\\
Then, in Proposition~\ref{prop:bc},
$z_2= -1$ and, for $c>2$, $z_c= (-1)^{c-1}$. 

Then, a generic monomial $Z= \prod_{\alpha j} \chi_{\alpha, j}^{b_{\alpha j}}$, for $b_{\alpha j} \in \mathbb{Z}$ belong to the center, i.e. it commutes with all minors  $\chi_{\beta, i}$, if and only if
$\sum_{\alpha j} \Lambda_{\beta i, \alpha j } b_{\alpha j}=0$, $\forall \beta, i$. From the discussion above, this is the case if and only if 
$$
\underline{b}=(b_{\alpha j})  \propto (0, \dots 0, 1| 0, \dots 0,-1|0, \dots 0, 1| \dots | 0, \dots 0, -1 | 1,-1, \dots , -1,1) \, .
$$
The result follows directly from this.
\end{proof}


\subsubsection{Dipper-Donkin; general case}

Set $n-1=xs$, $r+1=(x+z)s$, and $u+1=zs$. We assume that $n<r$ and that $s>1$ is the greatest common divisor of $n-1$ and $r+1$.
 We have (set $E_{i,j}=0$ if $i=0$ or $j=0$)
\begin{eqnarray*}
X^{n-1}&=&-\sum_{k=1}^{r}E_{u+1,k}+\sum_{j=0}^{s-1}\sum_{\ell=0}^{x-1}E_{zs+j+\ell s,j+\ell s}  +\sum_{j=0}^{s-1}\sum_{\ell=0}^{z-1}E_{j+\ell s,xs+j+\ell s} .
\end{eqnarray*}
We are looking for solutions to the equation $F\underline{v}=0$, or, equivalently, $X^{n-1}\underline{v}=\underline{v}$. We write $\underline{v}=\sum_{j=1}^rv_{j}{\bold e}_{j}$. Set
$$\underline{v}_i=\sum_{k=0}^{x
+z-1}{\bold e}_{i+ks}-\sum_{k=1}^{x+z-1}{\bold e}_{ks},\ ;\ i=1,\dots,s-1.$$
\\
 Due to the very explicit form of $X^{n-1}$ it follows  that 
$$\forall i=1,\dots, s-1: X^{n-1}\underline{v}_i=\underline{v}_i.$$
Furthermore, we clearly have a maximal, linearly independent set of solutions.
\\
Notice that the case $n=r$ odd corresponds to $s=2$. Hence we recover the previous result.
\\

As before, each solution vector $\underline{v}=\sum_{j=1}^rv_j{\bold e}_j$ results in a vector $\underline{b}=(\underline{b}_n,\dots,\underline{b}_2,\underline{b}_1)$, where $\underline{b}_1=\underline{v}$ such that $\Lambda\underline{b}={\bold 0}$.
It then follows from \eqref{one} and  \eqref{two} that \begin{eqnarray*}\underline{b}_2&=&-(\sum_{j=1}^rv_j){\bold e}_r\textrm{  and}\\
\underline{b}_c&=&\sum_j v_j\delta_{c-2,j~}{\bold e}_r= v_{c-2~}{\bold e}_r\textrm{ for }c\geq 2.\end{eqnarray*}

Similarly to Proposition~ \ref{prop:z}, we can now conclude that

\begin{prop} For $n\leq r$, and $s>1$ being the greatest common divisor of $r+1$ and $n-1$, the center of the {quasi-}Laurent polynomial algebra generated by the $n\times r$ elements $\chi_{\alpha,j} \in \mathcal{D}_{q^2}(M(n,r))$ is generated by the following $s-1$ elements:
\begin{equation*}
Z_i=\prod_{k=0}^{x-1}(\chi_{n-1-ks, r})^{-1}\prod_{k=0}^{x-1}\chi_{n-1-i-ks, r}\prod_{j=1}^{x+z-1}(\chi_{n, js})^{-1}\prod_{l=0}^{x+z-1}\chi_{n,i+ls}\ ; \quad i=1,\dots, s-1.
\end{equation*}
{This result recovers the previous one when $r=n$ is odd. }
\end{prop}

\subsubsection{FRT case} We sketch the similar result for the FRT algebra $\mathcal{O}_q(M(n,r))$. {Here we denote the minors $\chi_{\alpha,j}$ by  $\xi_{\alpha,j}$ in accordance with the notation in \eqref{minors}. }\\
Assume $n=xs$ and $r=ys$ with $x$ and $y$ both odd, and $s$ the greatest common divisor of $n,r$. Let $n \leq r$ (and thus $x \leq y$).

We again consider (\ref{H4}). First we must determine the kernel of the operator $F=F_S$. Equivalently, the $-1$ eigenspace of $S^{n}$. We easily have
\begin{equation*}S^n=-\sum_{i=1}^n E_{r-n+i,i}+\sum_{j=1}^{r-n}E_{j,n+j}.
\end{equation*}
 
We define \begin{equation}
\underline{v}_i=\sum_{\ell=0}^{y-1}(-1)^\ell{\bold e}_{i+\ell s}\ ; i=1,\dots,s\end{equation}
Using that $r-n=(y-x)s$ with $y-x$ even, one verifies
\begin{lemma}The vectors $v_1,\dots,v_s$ form a basis of the $-1$ eigenspace of $S^n$.
\end{lemma}
 In analogy with the Dipper-Donkin case we must next consider the vectors $\underline{b}^i=(\underline{b}^i_n,\dots,\underline{b}^i_2,\underline{b}_1^i)$ where,  $\forall c=2,\dots, n:\underline{b}^i_c=(X-T)X^{c-2}\underline{v}_i$, and $\underline{b}^i_1=\underline{v}_i$. In the present situation, $X\mapsto S=S_r$, and we easily get that $(X-T)X^{c-2}=-E_{r,c-1}$ for all $c=2,\dots,n$.

This results in $s$ elements of the kernel of $\Lambda$: 

\begin{prop} For $n\leq r$, the center of the {quasi-}Laurent polynomial algebra generated by the $n\times r$ elements $\xi_{\alpha,j} \in \mathcal{O}_q(M(n,r))$ of the FRT algebra is generated by the following $s$ elements (set $\xi_{r,0}=1$):
\begin{equation*}
\prod_{\ell=0}^{y-1}(\xi_{i+\ell s,n})^{(-1)^\ell}\prod_{k=0}^{x-1}(\xi_{r, n-i-ks})^{(-1)^{k+1}} ; \quad i=1,\dots,s.\end{equation*}
\end{prop}
\noindent

It is easy to see that we recover the result \cite[Lemma~4.1]{jj}.


\section{Compatible pairs}\label{se:cpairs}

Suppose $H$ is an $nr\times nr$ matrix as in \eqref{matrixH}. Suppose {there exist a matrix $K$ of order $nr$ such} that $KH=\left(\begin{array}{cc}\I_{nr-s} &Y\\0&0_s\end{array}\right)$, where $Y$ is an $(nr-s)\times s$ matrix for some non-negative  integer $s=\textrm{rank}( H)$ and the $0$ in the left corner denotes the $s \times (nr-s)$ zero matrix.  Consider 
$\Lambda_{a,b,d}={\mathbb T}_{a,b,d}^tH{\mathbb T}_{a,b,d}$ with ${\mathbb T}_{a,b,d}$ an invertible and  upper triangular matrix. Specifically, suppose 
${\mathbb T}_{a,b,d}=\left(\begin{array}{cc}a&b\\0&d\end{array}\right)$. Then 
$({\mathbb T}_{a,b,d})^{-1}=\left(\begin{array}{cc}a^{-1}&-a^{-1}bd^{-1}\\0&d^{-1}\end{array}
\right)$.
Hence,  $${\mathbb T}_{a,b,d}^{-1}K({\mathbb T}_{a,b,d}^t)^{-1}\Lambda_{a,b,d}={\mathbb T}_{a,b,d}^{-1}KH{\mathbb T}_{a,b,d}=\left(\begin{array}{cc}1&a^{-1}b+a^{-1}
Yd  \\0&0\end{array} \right).$$
\\
Set \begin{equation}\tilde B_{a,b,d}=2\cdot    \left({\mathbb T}_{a,b,d}^{-1}K({\mathbb T}_{a,b,d}^t)^{-1}\right)^t\end{equation}and let $B_{a,b,d}$ be the $nr\times (nr-s)$ matrix obtained from $\tilde B_{a,b,d}$ by removing the last $s$ columns. We then obtain

\begin{prop}
Suppose $b=-Yd$. Then  $(\Lambda_{a,b,d}, {B}_{a,b,d})$ is a compatible pair which satisfies \eqref{cc}.
\end{prop}

\medskip

\begin{rem}
This pair is of maximal rank and the non-mutable variables generate the center
of the algebra. In many situations it is natural to let more variables be
non-mutable. The most common choice is to let the $n+r-1$ covariant minors
$\chi_{n1}, \dots, \chi_{nr}, \dots , \chi_{1r}  $
be  non-mutable. A compatible pair
for the latter situation is of course easily obtained from the above by
truncation. 
\end{rem}
\noindent
It is obvious that we can write any ${\mathbb T}_{a,b,d}$ in terms of the already introduced matrix ${\mathbb T}$ \eqref{TT} as \begin{equation}{\mathbb T}_{a,b,d}={\mathbb T}\cdot \left(\begin{array}{cc}1&c_{a,b,d}\\0&1_s\end{array}\right)\end{equation} for some easily computed $(nr-s)\times s$ matrix $c_{a,b,d}$.  This then gives the change-of-basis needed to obtain the {cluster} variables   ${\mathcal V}^+_{a,b,d}$ for the compatible pair
{$(\Lambda_{a,b,d}, {B}_{a,b,d})$}  in terms of the variables ${\mathcal V}^+_{\mathcal M}$.

\bigskip
\subsection*{Acknowledgments}
The work of C.P is supported by the National Research Fund, Luxembourg, and cofunded under the Marie Curie Actions of the European Commission (FP7-COFUND).

\end{document}